\documentclass[11pt]{amsart}
\usepackage[utf8x]{inputenc}
\usepackage[english]{babel}
\usepackage[T1]{fontenc}
 
\usepackage{graphicx}
\usepackage{caption}
\usepackage{subcaption}
\usepackage{amssymb,amsmath}
\usepackage[hidelinks]{hyperref}
\usepackage{indentfirst}
\usepackage{enumerate,amsmath,amssymb, mathrsfs,mathtools}
\usepackage{appendix}
\usepackage{latexsym}
\usepackage{url}
\usepackage{color}
\usepackage{accents}
\usepackage{setspace}
\usepackage{pdfpages}
\usepackage{amsrefs}

\usepackage[margin=2.6cm]{geometry}
\allowdisplaybreaks

\newtheorem{prop}{Proposition}
\newtheorem{thm}[prop]{Theorem}
\newtheorem{lem}[prop]{Lemma}
\newtheorem{coro}[prop]{Corollary}
\newtheorem{rema}[prop]{Remark}

\BibSpec{article}{%
	+{}{\PrintAuthors} {author}
	+{,}{ } {title}
	+{, }{\textit } {journal}
	+{}{ \parenthesize} {date}
	+{,  }{no. } {volume}
	+{,}{ } {pages}
	+{,}{ } {note}
}
\ExplSyntaxOn

\ExplSyntaxOff
\BibSpec{book}{%
	+{}{\PrintAuthors}  {author}
	+{. }{}{title}
	+{,}{ }{series}
	+{,}{ vol.~}{volume}
	+{. }{\textit}{publisher}
}
\BibSpec{collection.article}{%
	+{}{\PrintAuthors}{author}
	+{, }{}{title}
	+{, }{\textit}{booktitle}
	+{, }{ \DashPages}{pages}
	+{,}{ }{series}
	+{, }{}{volume}
	+{, }{\textit}{publisher}
	+{,}{ }{date}
}

\title[Doubling of asymptotically flat half-spaces]{Doubling of asymptotically flat half-spaces and the Riemannian Penrose inequality}
\author{Michael Eichmair}
\address{
	\textnormal{Michael Eichmair \newline  \indent
		University of Vienna \newline \indent
		Faculty of Mathematics  \newline \indent
		Oskar-Morgenstern-Platz 1 \newline \indent
		1090 Vienna, 	Austria  \newline\indent 
		\href{https://orcid.org/0000-0001-7993-9536}{https://orcid.org/0000-0001-7993-9536} \newline\indent
		\href{mailto:michael.eichmair@univie.ac.atm}{michael.eichmair@univie.ac.at}}
}

\author{Thomas Koerber}
\address{\textnormal{Thomas Koerber  \newline \indent
		University of Vienna \newline \indent
		Faculty of Mathematics  \newline \indent
		Oskar-Morgenstern-Platz 1 \newline \indent 1090 Vienna,	Austria \newline\indent 
		\href{https://orcid.org/0000-0003-1676-0824}{https://orcid.org/0000-0003-1676-0824} \newline \indent 
		\href{mailto:thomas.koerber@univie.ac.atm}{thomas.koerber@univie.ac.at}}
}

\begin{document}

\date{\today}
\onehalfspacing

\begin{abstract}
Building on previous works of H.~L.~Bray, of P.~Miao, and of S.~Almaraz, E.~Barbosa, and L.~L.~de Lima, we develop a  doubling procedure for asymptotically flat half-spaces $(M,g)$ with horizon boundary $\Sigma\subset M$ and mass $m\in\mathbb{R}$.  If $3\leq \dim(M)\leq 7$, $(M,g)$ has non-negative scalar curvature, and the boundary $\partial M$ is mean-convex, we obtain the Riemannian Penrose-type inequality
$$
m\geq\left(\frac{1}{2}\right)^{\frac{n}{n-1}}\,\left(\frac{|\Sigma|}{\omega_{n-1}}\right)^{\frac{n-2}{n-1}}
$$
as a corollary. Moreover, in the case where $\partial M$ is not totally geodesic, we show how to construct local perturbations of $(M,g)$ that increase the scalar curvature. As a consequence, we show that equality holds in the above inequality if and only if the exterior region of $(M,g)$ is isometric to a Schwarzschild half-space. Previously, these results were  only known in the case where $\dim(M)=3$ and  $\Sigma$ is a connected free boundary hypersurface. 
\end{abstract}

\maketitle

\section{Introduction}
Let $(M,g)$ be a  connected, complete  Riemannian manifold of dimension $3\leq n\leq 7$ with integrable scalar curvature $R(g)$ and non-compact boundary $\partial M$ with integrable mean curvature $H(\partial M,g)$. Here, $H(\partial M,g)$ is computed as the divergence along $\partial M$  of the  normal $-\nu(\partial M,g)$ pointing out of $M$. \\ \indent  We say that $(M,g)$ is an asymptotically flat half-space if there is  a number $\tau>(n-2)/2$ and a non-empty compact subset of $M$ whose complement is diffeomorphic to $\{x\in\mathbb{R}^n_+:|x|_{\bar g}>1\}$ such that, in this so-called asymptotically flat chart, as $x\to\infty$,
\begin{align} \label{af intro} 
|g-\bar g|_{\bar g}+|x|_{\bar g}\,|D(\bar g)g|_{\bar g}+|x|^2_{\bar g}\,|D^2(\bar g)g|_{\bar g}=O(|x|_{\bar g}^{-\tau}).
\end{align} 
Here,  $\mathbb{R}^n_+=\{x\in\mathbb{R}^n:x^n\geq 0\}$ is the upper half-space and $\bar g$ the Euclidean metric.
\\ \indent 
J.~Escobar has studied asymptotically flat half-spaces in the context of the Yamabe problem for compact Riemannian manifolds with boundary; see \cite{Escobar} and also the related works of S.~Brendle \cite{BrendleAsian} and of S.~Brendle and S.-Y.~S. Chen \cite{BrendleChen}. S.~Almaraz \cite[pp.~2628-2629]{Almaraz} and S.~Almaraz, E.~Barbosa, and L.~L.~de Lima  \cite[p.~674]{ABDL} have studied asymptotically flat half-space in detail and associated to them a global geometric invariant called the mass. This mass,  whose definition is attributed to F.~C.~Marques on \cite[p.~677]{ABDL}, is given by 
\begin{equation} \label{half-space mass intro}
\begin{aligned}   
m(g)= \frac{1}{2\,(n-1)\,\omega_{n-1}}\,\lim_{\lambda\to\infty}\lambda^{-1}\,\bigg(\sum_{i,\,j=1}^n&\int_{ \mathbb{R}^n_+\cap {S}^{n-1}_\lambda(0)}x^i\,\big[(\partial_jg)(e_i,e_j)-(\partial_ig)(e_j,e_j)\big]\,\mathrm{d}\mu(\bar g)\\&+\sum_{i=1}^{n-1}\int_{(\mathbb{R}^{n-1}\times\{0\})\cap S^{n-1}_\lambda(0)} x^i\,g(e_i,e_n)\,\mathrm{d}l(\bar g) \bigg )
\end{aligned}
\end{equation} 
where the integrals are computed in the asymptotically flat chart. Here, $\omega_{n-1}=|\{x\in\mathbb{R}^n:|x|_{\bar g}=1\}|_{\bar g}$ denotes the Euclidean area of the $(n-1)$-dimensional unit sphere and $e_1,\dots,e_n$ is the standard basis of $\mathbb{R}^n$. In analogy with the work \cite{SchoenYamabe} of  R.~Schoen on closed manifolds, J.~Escobar  has established a connection between the magnitude of the Yamabe-invariant of a compact manifold with boundary and the sign of the mass \eqref{half-space mass intro} of an associated asymptotically flat half-space in \cite{Escobar}.  S.~Almaraz, E.~Barbosa, and L.~L.~de Lima have showed that the mass \eqref{half-space mass intro} is a geometric invariant and, in fact, non-negative provided that $(M,g)$ satisfies suitable energy conditions. As noted in \cite[p.~675]{ABDL}, previous results in this direction had been obtained by J.~Escobar  \cite[Appendix]{Escobar} and by S.~Raulot \cite[Theorem 23]{Raulot}.
\begin{thm}[{\cite[Theorem 1.3]{ABDL}}] \label{pmt half space}
Let $(M,g)$ be an asymptotically flat half-space of dimension $ 3\leq n\leq7$ such that $R(g)\geq 0$ and $H(\partial M,g)\geq 0$. Then $m(g)\geq 0$. Moreover, $m(g)=0$ if and only if $(M,g)$ is isometric to $(\mathbb{R}^n_+,\,\bar g)$.
\end{thm}
\begin{rema}
As explained in \cite[\S2]{Miao}, the assumption that $H(\partial M,g)\geq 0$ can and should be viewed as a non-negativity condition for the scalar curvature $R(g)$ across $\partial M$ in a distributional sense; see also \cite[p.~207]{Bray}. We note that this condition also has a natural physical interpretation; \cite[Remark 2.7]{Luciano}.  \label{Miao remark} 
\end{rema}
Theorem \ref{pmt half space} is fashioned after the positive mass theorem for asymptotically flat initial data for the Einstein field equations, which has been proved   by R.~Schoen and S.-T.~Yau \cite{SchoenYau} using minimal surface techniques and by E.~Witten \cite{Witten} using certain solutions of the Dirac equation. In the presence of a so-called outermost minimal surface in the initial data set, a heuristic argument due to R.~Penrose \cite{Penrosenaked} suggests a stronger, quantitative version of the positive mass theorem which has been termed the Riemannian Penrose inequality. This inequality has been verified by G.~Huisken and T.~Ilmanen in \cite{HI} in dimension $n=3$ when the outermost minimal surface is connected, by H.~L.~Bray in \cite{Bray} in the case of a possibly disconnected outermost minimal surfaces, and by H.~L.~Bray and D.~Lee in \cite{BrayLee} in the case where $3\leq n\leq7$ and  the outermost minimal surface may be disconnected. We provide more details on asymptotically flat manifolds without boundary, the positive mass theorem, and the Riemannian Penrose inequality in Appendix \ref{af appendix}. \\ \indent 
 S.~Almaraz, L.~De Lima, and L.~Mari \cite{Luciano} have studied the mass \eqref{half-space mass intro} of  initial data sets with  non-compact boundary in a spacetime setting; see also the recent survey \cite{deLima} of L.~De Lima.  Moreover, they argue that, in the presence of an outermost minimal surface, a  Riemannian Penrose-type inequality should  hold for asymptotically flat half-spaces as well; see \cite[Remark 5.6]{Luciano}.  \\ \indent 
 	\begin{figure}\centering
 	\includegraphics[width=0.7\linewidth]{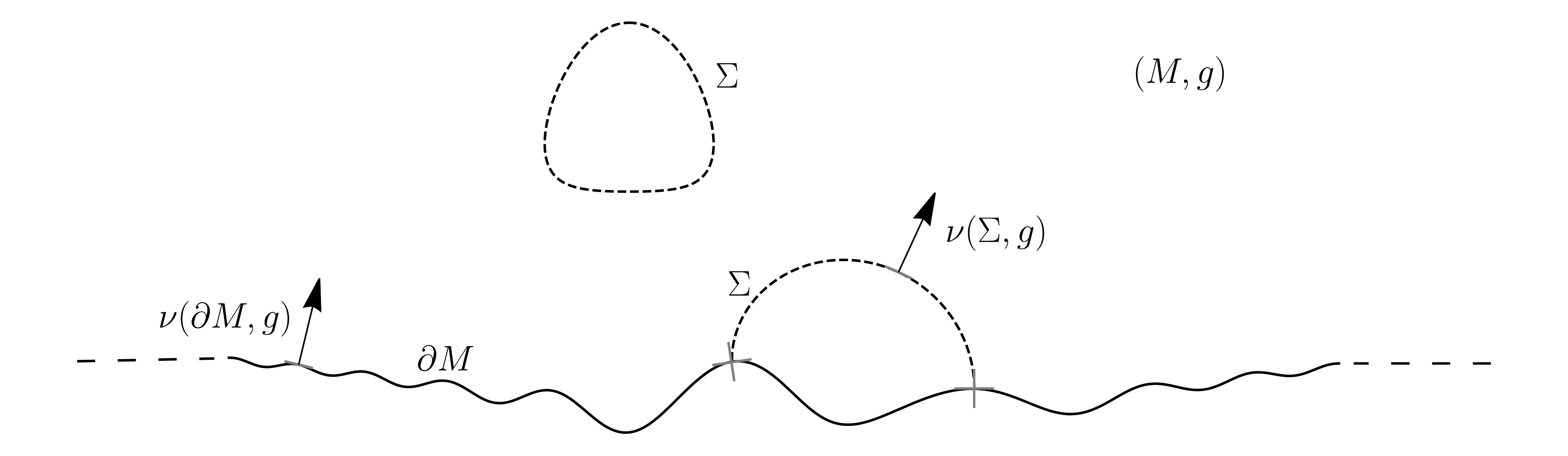}
 	\caption{An illustration of an asymptotically flat half-space $(M,g)$ with horizon boundary. $\partial M$ is illustrated by the solid black line. The horizon boundary $\Sigma$, consisting of a free boundary hypersurface and a closed hypersurface, is illustrated by the two dotted lines. The unit normals $\nu(M,g)$ and $\nu(\Sigma,g)$ are showed by the two arrows. }
 	\label{fb_closed}
 \end{figure}
  To describe recent results in this direction, we recall the following definitions from \cite{thomas_penrose}; see Figure \ref{fb_closed}. Let $\Sigma\subset M$ be a  compact separating  hypersurface satisfying $\Sigma\cap \partial M=\partial \Sigma$ with normal $\nu(\Sigma,g)$ pointing towards the closure $M(\Sigma)$ of the non-compact component  of $M\setminus \Sigma$.  We call a component $\Sigma^0$ of $\Sigma$  a free boundary hypersurface if $\partial \Sigma^0\neq \emptyset$ and $\nu(\Sigma)(x)\in T_x\partial M$ for every $x\in\partial \Sigma^0$. If $\partial \Sigma^0=\emptyset$, we call $\Sigma^0$ a closed hypersurface.  \\ \indent We say that an asymptotically flat half-space $(M,g)$ has horizon boundary $\Sigma\subset M$ if $\Sigma$ is a non-empty compact  minimal hypersurface with the following properties.
 \begin{itemize}
 	\item[$\circ$] The connected components of $\Sigma$ are either free boundary hypersurfaces or closed hypersurfaces.
 	\item[$\circ$] Every    minimal  free boundary hypersurfaces or minimal closed  hypersurfaces in $M(\Sigma)$ is a component of $\Sigma$.
 \end{itemize}
\begin{rema}
	The proof of \cite[Lemma 2.3]{thomas_penrose} shows that every asymptotically flat half-space $(M,g)$ of dimension $3\leq n\leq7$ with $H(\partial M,g)\geq 0$ either has  a  unique horizon boundary $\Sigma\subset M$ or contains no compact  minimal hypersurfaces. 
\end{rema}
 The region $M(\Sigma)$ outside of the horizon boundary is called an exterior region. The horizon $\Sigma$ is also called an outermost minimal surface.  An example of an exterior region with horizon boundary is the Schwarzschild half-space of mass $m>0$ and dimension $n\geq 3$, given by
 \begin{align} \label{Schwarzschild half space} 
 (M(\Sigma),g)=\bigg(\left\{x\in\mathbb{R}^n_+:|x|_{\bar g}\geq m^{\frac{1}{n-2}}\right\},\left(1+m\,|x|_{\bar g}^{2-n}\right)^\frac{4}{n-2}\,\bar g\bigg)
 \end{align}
where $\Sigma=\{x\in\mathbb{R}^n_+:|x|_{\bar g}=m^{\frac{1}{n-2}}\}$;
see \cite[Remark 3.10]{ABDL} and \cite[\S2]{thomas_penrose}. \\ \indent
The second-named author has recently proved the following Riemannian Penrose-type inequality for asymptotically flat half-spaces whose horizon boundary is a connected free boundary hypersurface.  
\begin{thm}[{\cite[Theorem 1.2]{thomas_penrose}}] \label{thomas penrose} Let $(M,g)$ be an asymptotically flat half-space of dimension $n=3$ with horizon boundary $\Sigma\subset M$ such that the following three conditions hold. 
	\begin{itemize}
		\item[$\circ$] $R(g)\geq 0$ in $M(\Sigma)$.
		\item[$\circ$] $H(\partial M,g)\geq 0$ on $ M(\Sigma)\cap\partial M$.
		\item[$\circ$] $\Sigma$ is a connected free boundary hypersurface. 
	\end{itemize}
Then 
$$
m(g)\geq\sqrt{\frac{|\Sigma|_g}{32\,\pi}}
$$ 
with equality if and only if $(M(\Sigma),g)$ is isometric to a Schwarzschild half-space \eqref{Schwarzschild half space}. 
\end{thm}
\begin{rema}
Previous results in direction of Theorem \ref{thomas penrose} for asymptotically flat half-spaces arising as certain graphical hypersurfaces in Euclidean space had been obtained by E.~Barbosa and A.~Meira  \cite{BarbosaMeira}.
\end{rema}
\begin{rema}
The method of weak free boundary inverse mean curvature flow employed in the proof of Theorem \ref{thomas penrose} in \cite{thomas_penrose} had been studied previously by T.~Marquardt in \cite{Marquardt}. It appears to the authors of this paper that the scope of this method  is essentially limited to the case where $n=3$ and $\Sigma$ is a connected free boundary hypersurface; see \cite[p.~16]{thomas_penrose}.
\end{rema}
\begin{rema}
	Theorem \ref{thomas penrose} is related to a Penrose-type inequality for so-called asymptotically flat support surfaces conjectured by G.~Huisken and studied by A.~Volkmann; see \cite[p.~38]{volkmann2015free} and \cite[Lemma 2.1]{thomas_penrose}.
\end{rema}

\subsection*{Outline of our results}
\indent Comparing Theorem \ref{thomas penrose} with the Riemannian Penrose inequality for asymptotically flat manifolds, stated here as Theorem \ref{RPI no boundary}, suggests that the assumptions that $n=3$ and that $\Sigma$ be a connected  free boundary hypersurface in Theorem \ref{thomas penrose} are not necessary. The goal of this paper is to address this conjecture using a strategy different from  that in \cite{thomas_penrose}. In fact, we demonstrate how the gluing method developed by P.~Miao in \cite{Miao}, which in turn expands on an idea of H.~L.~Bray \cite{Bray}, can be used to develop  a doubling procedure for asymptotically flat half-spaces that reduces the Riemannian Penrose inequality for asymptotically flat half-spaces to the Riemannian Penrose inequality  for asymptotically flat manifolds. \\ \indent
For the statement of Theorem \ref{main thm}, recall from Appendix \ref{af appendix} the definition of an asymptotically flat manifold  $(\tilde M,\tilde g)$, of its mass $\tilde m$, of its horizon boundary $\tilde \Sigma$, and of the exterior region $\tilde M(\tilde \Sigma)$.
\begin{thm} \label{main thm}
Let $(M,g)$ be an asymptotically flat half-space of dimension $3\leq n\leq 7$ with horizon boundary $\Sigma\subset M$ such that the following two conditions hold.
\begin{itemize}
	\item[$\circ$] $R(g)\geq 0$ in $M(\Sigma)$.
	\item[$\circ$] $H(\partial M,g)\geq 0$ on $ M(\Sigma)\cap\partial M$.
\end{itemize}
Let $\varepsilon>0$. There exists an asymptotically flat manifold $(\tilde M,\tilde g)$ with horizon boundary $\tilde\Sigma\subset \tilde M$ such that
\begin{itemize}
	\item[$\circ$] $R(\tilde g)\geq 0$ in $\tilde M(\tilde \Sigma)$,
	\item[$\circ$] $|\tilde m(\tilde g)- 2\,m(g)|< \varepsilon$, and
	\item[$\circ$] $||\tilde \Sigma|_{\tilde g}-2\,|\Sigma|_g|< \varepsilon$. 
\end{itemize}
\end{thm}
\begin{rema}
	Gluing constructions related to the one used in the proof of Theorem \ref{main thm}  have also been studied by P.~Miao and S.~McCormick in \cite{MiaoCormick} and by S.~Lu and P.~Miao in \cite{LuMiao}.
\end{rema}
Combining Theorem \ref{main thm} with Theorem \ref{RPI no boundary}, we are able to extend the Riemannian Penrose inequality for asymptotically flat half-spaces to dimensions less than $8$ and to horizon boundaries that may be disconnected. 
\begin{coro} \label{penrose coro} 
Let $(M,g)$ be an asymptotically flat half-space of dimension $3\leq n\leq 7$ with horizon boundary $\Sigma\subset M$ such that the following two conditions hold. 
\begin{itemize}
	\item[$\circ$] $R(g)\geq 0$ in $M(\Sigma)$.
	\item[$\circ$] $H(\partial M,g)\geq 0$ on $ M(\Sigma)\cap\partial M$.
\end{itemize}
Then 
\begin{align} \label{coro RPI} 
m(g)\geq\left(\frac{1}{2}\right)^{\frac{n}{n-1}}\,\left(\frac{|\Sigma|_g}{\omega_{n-1}}\right)^{\frac{n-2}{n-1}}.
\end{align} 
\end{coro}
\begin{rema} A.~Carlotto and R.~Schoen have showed in \cite[Theorem 2.3]{CarlottoSchoen}
that there is an abundance of asymptotically flat Riemannian manifolds with non-negative scalar curvature that contain a Euclidean half-space isometrically. Note that Corollary \ref{penrose coro} shows that the Riemannian Penrose inequality, stated here as Theorem \ref{RPI no boundary}, can be localized to the geometrically non-trivial part of such initial data.
\end{rema}
The approximation argument used  to prove Corollary \ref{penrose coro} cannot be applied to characterize the case of equality in \eqref{coro RPI} directly. Yet, we observe that $(M,g)$ can be locally perturbed to increase the scalar curvature near non-umbilical points of the boundary $\partial M$. Combining this insight with a variational argument used by R.~Schoen and S.-T.~Yau \cite{SchoenYau} to characterize the case of equality in the positive mass theorem, we are able to prove the following rigidity result. 
\begin{thm} \label{rigidity remark}
Let $(M,g)$ be an asymptotically flat half-space of dimension $ 3\leq n\leq7$ with horizon boundary $\Sigma\subset M$ such that the following two conditions hold. 
\begin{itemize}
	\item[$\circ$] $R(g)\geq 0$ in $M(\Sigma)$.
	\item[$\circ$] $H(\partial M,g)\geq 0$ on $ M(\Sigma)\cap\partial M$.
\end{itemize}
Assume that 
\begin{align*} m(g)=\left(\frac{1}{2}\right)^{\frac{n}{n-1}}\,\left(\frac{|\Sigma|_g}{\omega_{n-1}}\right)^{\frac{n-2}{n-1}}.
\end{align*} 
Then $(M(\Sigma),g)$ is isometric to a Schwarzschild half-space \eqref{Schwarzschild half space}.
\end{thm}  
\begin{rema} 
	H.~L.~Bray and D.~A.~Lee \cite{BrayLee} have proved rigidity of the Riemannian Penrose inequality for asymptotically flat manifolds $(\tilde M,\tilde g)$ of dimension $3\leq n\leq7$ under the additional assumption that $(\tilde M,\tilde g)$ be spin; see Theorem \ref{RPI no boundary}. Building on previous work \cite{McFeron} by D.~McFeron and G.~Sz\'{e}kelyhidi, S.~Lu and P.~Miao \cite[Theorem 1.1]{LuMiao} have showed that the spin assumption  can be dispensed with. 	Using the techniques developed in this paper, we are able to give a short alternative proof of this fact; see Theorem \ref{RPI rigidity}.
\end{rema}
\begin{rema}
	We survey several important contributions to scalar curvature rigidity results preceding Theorem \ref{rigidity remark} in Appendix \ref{scalar appendix}. 
\end{rema}
\begin{rema}
	For the proofs of Theorem \ref{main thm}, Corollary \ref{penrose coro}, and Theorem \ref{rigidity remark}, it is sufficient to require the metric $g$ to be of class $C^{2,\alpha}$. For the sake of readability, we will assume throughout that $g$ is smooth. 
\end{rema}
\subsection*{Outline of the proof}
Let $(M,g)$ be an asymptotically flat half-space with horizon boundary $\Sigma\subset M$ and suppose that $R(g)\geq 0$ in $M(\Sigma)$ and $H(\partial M,g) \geq 0$ on $ M(\Sigma)\cap\partial M$. The basic idea to prove Theorem \ref{main thm} is to consider the double $(\tilde M, \tilde g)$ of $(M,g)$   obtained by reflection across $\partial M$. The metric $\tilde g$ is only $C^0$ across $\partial M$. The condition $H(\partial M,g)\geq 0$ suggests that the scalar curvature of $\tilde g$ is non-negative in a distributional sense; see Remark \ref{Miao remark}. Moreover, since $\Sigma\subset M$ is an outermost minimal surface that intersects $\partial M$ orthogonally,  its double $\tilde \Sigma\subset \tilde M$ is  an outermost minimal surface without boundary. \\ \indent
The difficulty in rendering this heuristic argument rigorous is that $(\tilde M,\tilde g)$ needs to be smoothed in a way that allows us to keep track of the mass, the horizon boundary, and the relevant energy conditions all at the same time. To this end, we first adapt an approximation procedure developed by  S.~Almaraz, E.~Barbosa, and L.~L.~de Lima in  \cite[Proposition 4.1]{ABDL} to arrange that $g$ is scalar flat and conformally flat at infinity and that $\partial M$ is totally geodesic at infinity; see Proposition \ref{ahf to conformally flat prop}. In particular, the reflected metric $\tilde g$ is $C^2$ outside of a bounded open set $W\subset M$. Moreover, using a local conformal perturbation of the metric, we may arrange that $\Sigma$ is strictly mean convex; see Lemma \ref{strictly mean convex}. \\ \indent
Next, we smooth $\tilde g$ near $ W\cap\partial M$ using a technique developed by P.~Miao in \cite{Miao}. In this step, the mean convexity of $\partial M$ ensures that the scalar curvature of the smoothed metric remains uniformly bounded from below near $\partial M$; see Lemma \ref{scalar curvature estimate from below}. Moreover, we show that the strict mean convexity of $\tilde \Sigma$ is not affected by this  procedure; see Lemma \ref{strictly mean convex delta}. \\ \indent  By a conformal transformation similar to that developed by P.~Miao in \cite[\S4]{Miao}, building in turn on \cite[Lemma 3.3]{SchoenYau}, we remove the small amount of negative scalar curvature that may have been created close to $\partial M$ in the approximation process. This conformal transformation only changes the mass of the smoothed manifold by a small amount; see Proposition \ref{conf change prop}. Finally, using $\tilde \Sigma$ as a barrier, it follows that the smoothed metric has horizon boundary. Since $\Sigma\subset M$ is area-minimizing, it follows that the area of the new horizon boundary is at least as large as that of $\tilde \Sigma$; see Lemma \ref{horizon boundary hat }. This is how we obtain Theorem \ref{main thm}. \\ \indent
To prove Theorem \ref{rigidity remark}, we first construct a global conformal perturbation of $(M(\Sigma),g)$ that preserves the conditions $R(g)\geq 0$ and $H(\partial M,g)\geq 0$, strictly decreases $m(g)$  unless $R(g)=0$, and which changes the area of $\Sigma$ only marginally. Second, if the second fundamental form $h(\partial M,g)$ of $\partial M$ does not vanish, we construct a local  perturbation of $(M(\Sigma),g)$ that increases $R(g)$, preserves the condition $H(\partial M,g)\geq 0$, and changes neither $m(g)$ nor $|\Sigma|_g$. We note that a perturbation with these properties could not possibly be conformal; it has to be fine-tuned to the geometry of $\partial M$.  Consequently, if equality in \eqref{coro RPI} holds, then $\partial M$ is totally geodesic and the double $(\tilde M,\tilde g)$ is $C^2$-asymptotically flat. Theorem \ref{rigidity remark} now follows from Theorem \ref{RPI no boundary}.
  \subsection*{Acknowledgments}
 The authors acknowledge the support of the START-Project Y963 of the Austrian Science Fund. The second-named author acknowledges the support of the Lise-Meitner-Project M3184 of the Austrian Science Fund. The authors thank Pengzi Miao for helpful feedback on the statement of Theorem \ref{RPI rigidity} and for bringing the results in \cite{LuMiao} to their attention. The authors thank the anonymous referees for their feedback which has improved the exposition of this paper. \\ \indent This paper is dedicated to the memory of Robert Bartnik.

\section{Reduction to conformally flat ends}
In this section, we assume that $(M, g)$ is an asymptotically flat half-space of dimension $3 \leq n \leq 7$ and decay rate $\tau>(n-2)/2$. We also assume that $(M,g)$ has horizon boundary $\Sigma\subset M$ and that $R(g)\geq 0$ in $M(\Sigma)$ and $H(\partial M, g)\geq 0$ on $ M(\Sigma)\cap\partial M$. \\ \indent The goal of this section is to  approximate the Riemannian metric  $g$ by a sequence $\{g_i\}_{i=1}^\infty$ of Riemannian metrics $g_i$ on $M$ that are scalar flat, conformally flat, and such that $h(\partial M, g_i)=0$ outside of some compact set. \\ \indent  Here and below, $\Sigma$ and $\partial M$ are oriented by their unit normal vectors $\nu(\Sigma,g)$ and $\nu(\partial M,g)$ pointing towards $M(\Sigma)$.  $H(\Sigma,g)$ and $H(\partial M,g)$ are  computed as the divergence of $-\nu(\Sigma,g)$ along $\Sigma$ and the divergence of $-\nu(\partial M,g)$ along $\partial M$, respectively.
\begin{prop} \label{ahf to conformally flat prop} 
Let $\tau'\in\mathbb{R}$ be such that $(n-2)/2<\tau'<\tau$. There exist sequences $\{g_i\}_{i=1}^\infty$ of Riemannian metrics $g_i$ on $M$ and $\{K_i\}_{i=1}^\infty$ of compact sets $K_i\subset M$ such that $(M, g_i)$ is an asymptotically flat half-space with horizon boundary $\Sigma_i\subset M(\Sigma)$ and such that the following properties hold.
\begin{itemize}
	\item[$\circ$] $(M,g_i)$ is conformally flat in $M\setminus K_i$.
	\item[$\circ$] $R(g_i)=0$ in $M\setminus K_i$.
	\item[$\circ$] $h(\partial M,g_i)=0$ on $\partial M\setminus K_i$.
	\item[$\circ$] $R(g_i)\geq 0$ in $M(\Sigma)$.
	\item[$\circ$] $H(\partial M,g_i)\geq 0$ on $ M(\Sigma)\cap\partial M$.
	\item[$\circ$] $m(g_i)=m(g)+o(1)$ as $i\to\infty$.
	\item[$\circ$] $|\Sigma_i|_{g_i}=|\Sigma|_g+o(1)$ as $i\to\infty$.
	\item[$\circ$] $g_i\to g$ in $C^0(M)$ as $i\to\infty$.
\end{itemize}
Moreover, 
\begin{align} \label{uniform hemisphere} 
	\sup_{i\geq 1}\,\limsup_{x\to\infty}
	\left[|x|_{\bar g}^{\tau}\,|g_i-\bar g|_{\bar g}+|x|^{\tau+1}_{\bar g}\,|D(\bar g)g_i|_{\bar g}+|x|^{\tau+2}_{\bar g}\,|D^2(\bar g)g_i|_{\bar g}\right]<\infty.
\end{align} 

\end{prop}
\begin{proof}
 Arguing as in the proof of \cite[Proposition 4.1]{ABDL} but using Proposition \ref{PDE prop} instead of \cite[Proposition 3.3]{ABDL}, we obtain a sequence  $\{g_i\}_{i=1}^\infty$ of Riemannian metrics $g_i$ on $M$ and a sequence $\{K_i\}_{i=1}^\infty$ of compact sets $K_i\subset M$ exhausting $M$ such that the following properties hold.
\begin{itemize}
	\item[$\circ$] $g_i$ is conformally flat  in $M\setminus K_i$.
		\item[$\circ$] $R(g_i)=0$ in $ M\setminus K_i$.

	\item[$\circ$] $h(\partial M,g_i)=0$ on $\partial M\setminus K_i$.
	\item[$\circ$] $R(g_i)\geq 0$ in $M(\Sigma)$.
	\item[$\circ$] $H(\partial M,g_i)\geq 0$ on $ M(\Sigma)\cap\partial M$.
	\item[$\circ$] $H(\Sigma,g_i)=0$ on $\Sigma$.
		\item[$\circ$] $m(g_i)=m(g)+o(1)$ as $i\to\infty$.
\end{itemize}
 Moreover,  $g_i\to g$ in $C^0(M)$, $g_i\to g$  in $C_{loc}^2(M)$, and \eqref{uniform hemisphere} holds. \\
  \indent 
By \eqref{uniform hemisphere}, there is $\lambda_0>1$ such that the hemispheres $\mathbb{R}^n_+\cap S_\lambda(-1/2\,\lambda\,e_n)$ have negative mean curvature with respect to $g_i$ and meet $\partial M$ at an acute angle  with respect to $g_i$ provided that $\lambda \geq \lambda_0$ and $i$ is  sufficiently large. We consider the class of all embedded hypersurfaces of $M(\Sigma)$ that are homologous to $\Sigma$ in $M(\Sigma)$  and whose boundary is  contained in $\partial M$  and homotopy equivalent to $\partial \Sigma$ in $ M(\Sigma) \cap \partial M$. Since $H(\Sigma,g_i)=0$, it follows from \cite[Theorem 1]{meeks1982existence} that there is an  outermost  minimal hypersurface $ \Sigma_i\subset M(\Sigma)$ that is homologous to $\Sigma$ in $M(\Sigma)$, whose boundary is homotopy equivalent to $\partial \Sigma$ in $ M(\Sigma)\cap\partial M$, and whose components are either free boundary hypersurfaces or closed hypersurfaces. Moreover,
\begin{align} \label{area estimate} 
\sup_{i\geq 1} |\Sigma_i|_{g_i}<\infty.
\end{align}  Recall from \cite[Lemma 2.3]{thomas_penrose} that $\Sigma$ is area-minimizing with respect to $g$ in $M(\Sigma)$. Consequently, as $i\to\infty$,
	$$
	|\Sigma|_{g}\leq|\Sigma_i|_{g}\leq (1+o(1))| \Sigma_i|_{g_i}.  
	$$
Finally, using \eqref{area estimate}, the curvature estimate \cite[Lemma 3.3]{thomas_penrose}, and standard elliptic theory, it follows that, passing to another subsequence if necessary, $\{\Sigma_i\}_{i=1}^\infty$ converges to a minimal surface $\Sigma_0\subset  M(\Sigma)$ with respect to $g$ in $C^{1,\alpha}(M)$ and smoothly away from $\partial M$, possibly with finite multiplicity. Since $M(\Sigma)$ is an exterior region, it follows that $\Sigma_0=\Sigma$. Since $\Sigma_i$ is area-minimizing in $M(\Sigma_i)$ with respect to $g_i$, we obtain that, as $i\to\infty$,
	$$
|\Sigma_i|_{g_i}\leq (1+o(1)) |\Sigma|_{g_i}\leq (1+o(1))\,|\Sigma|_{g}.  
$$
	\indent The assertion follows.
\end{proof} 
\section{Gluing of asymptotically flat half-spaces} \label{gluing section}
In this section, we assume that $(M,g)$ is an asymptotically flat half-space of dimension $3\leq n\leq 7$ with horizon boundary $\Sigma\subset M$ such that the following properties are satisfied. 
\begin{itemize}
	\item[$\circ$] $(M,g)$ is conformally flat outside of a compact set. 
	\item[$\circ$] $R(g)=0$ outside of a compact set.
	\item[$\circ$] $h(\partial M, g)=0$ outside of a compact set.
	\item[$\circ$] $H(\partial M,g)\geq 0$ on $ M(\Sigma)\cap\partial M$.
\end{itemize}
\indent \indent The goal of this section is to double $(M,g)$ by reflection across $\partial M$ and  to appropriately smooth the metric of the double.

\begin{figure}\centering
	\includegraphics[width=0.5\linewidth]{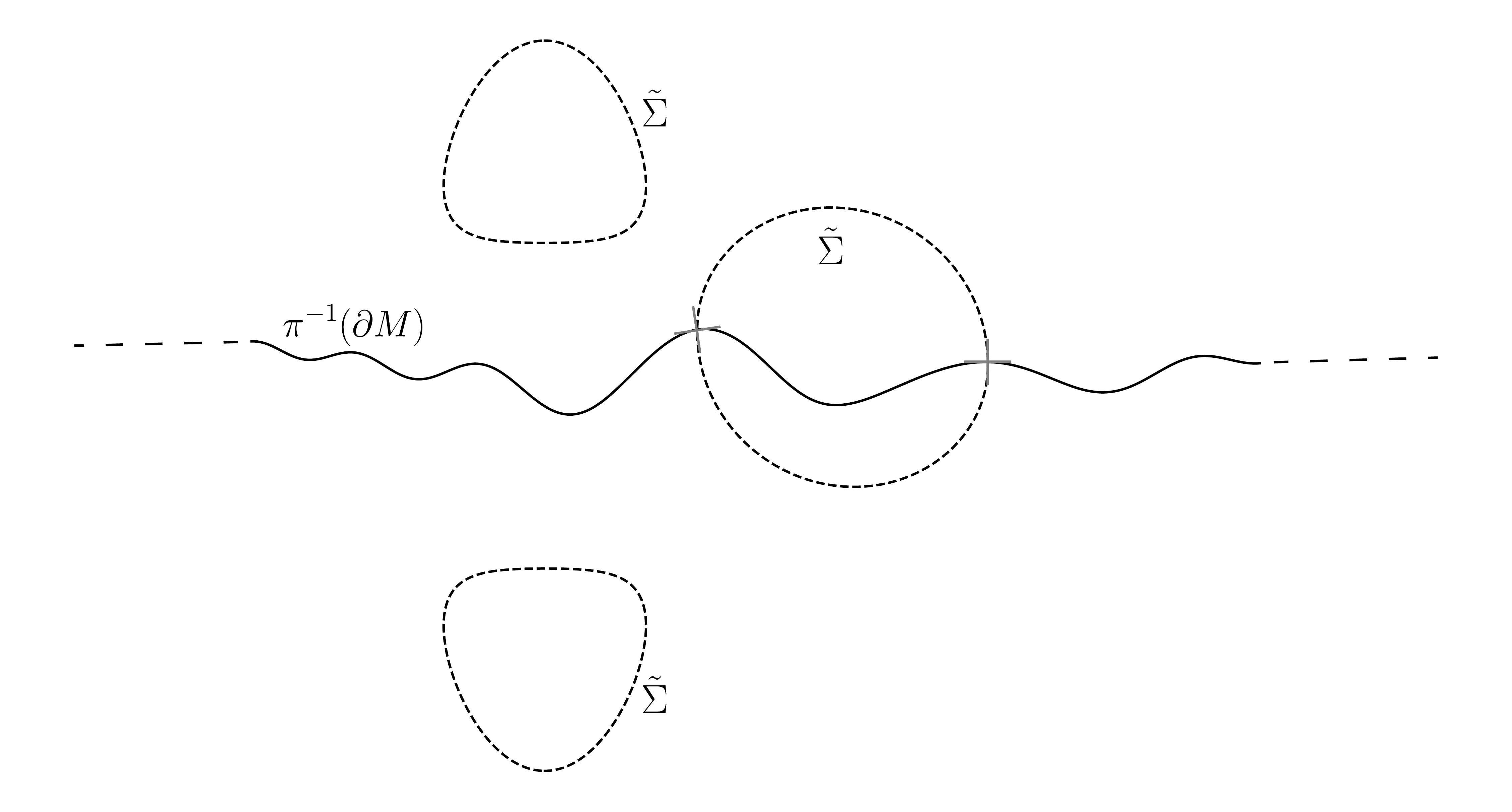}
	\caption{An illustration of the double $(\tilde M,\tilde g)$ of $(M,g)$. $(\tilde M,\tilde g)$ is obtained by reflection across $\partial M$ which is illustrated by the solid black line. Here, $\Sigma$ has two components while $\tilde \Sigma$ has three components illustrated by the dashed lines.  }
	\label{double}
\end{figure}
\begin{lem}
	There is $\delta_0>0$ with the following property. The map \label{Phi lemma}
	$$
	\Phi:\partial M\times[0,\delta_0)\to\{x\in M:\operatorname{dist}(x,\partial M,g)<\delta_0\}$$ given by 
	$$\Phi(y,t)=\exp(g)_y(t\,\nu(\partial M,g)(y))
	$$ 
	is a diffeomorphism. 
\end{lem}
\begin{proof}
Clearly, $\Phi$ is a local diffeomorphism and surjective. Moreover, by compactness and Lemma \ref{geodesic lemma}, using the fact that $g$ is asymptotically flat \eqref{af intro}, it follows that $\Phi$ is injective.   
\end{proof}
\begin{rema}
	It follows from Lemma \ref{Phi lemma} that there is a  smooth family $\{\gamma(g)_t:t\in[0,\delta_0)\}$ of Riemannian metrics $\gamma(g)_t$ on $\partial M$ such that
	$$
	\Phi^*g=\gamma(g)_t+dt^2.
	$$ 
\end{rema}
Let \begin{align} \tilde M=(M\times\{-1,1\})/\sim \label{tilde M}\end{align} where
\begin{itemize}
	\item[$\circ$] $(x_1,\pm1)\sim(x_2,\pm1)$ if and only if $x_1=x_2$  and
	\item[$\circ$] $(x_1,\pm1)\sim(x_2,\mp1)$ if and only if $x_1,\,x_2\in \partial M$ and  $x_1=x_2.$
\end{itemize}   Let \begin{align} \label{projection} \pi:\tilde M\to M\qquad\text{be given by}\qquad\pi([(x,\pm1)])=x\end{align}  and $$\tilde \Sigma=\pi^{-1}(\Sigma);$$ see Figure \ref{double}.
  \\ \indent  We consider the map
$$
\tilde \Phi:\partial M\times(-\delta_0,\delta_0)\to\{\tilde x\in \tilde M:\operatorname{dist}(\pi(\tilde x),\partial M,g)<\delta_0\}$$ given by
$$\qquad \tilde \Phi(y,t)=[(\Phi(y,|t|),\operatorname{sign}(t))].
$$
We obtain a smooth structure on $\tilde M$ by requiring that the map $\tilde \Phi$ be smooth. Moreover, 
given $t\in(-\delta_0,0]$, we define $\gamma(g)_{t}=\gamma(g)_{-t}$. 
Note that $\gamma(g)_t+dt^2$ is continuous on $\partial M\times(-\delta_0,\delta_0)$. It follows that the Riemannian metric $\tilde g$ on $\tilde M$ defined by
$$\tilde g=\pi^*g$$ 
is continuous across $\partial M$. \\ \indent
For the following lemma, recall from Appendix \ref{af appendix} the definitions \eqref{asymptotically flat mf} of an asymptotically flat metric  and \eqref{mass} of the mass of an asymptotically flat manifold without boundary. 
\begin{lem} \label{mass lemma} 
$\tilde g$ is of class $C^0$ and $C^2$-asymptotically flat. Moreover, the following properties  hold.
\begin{itemize}
\item[$\circ$]	$
\tilde m(\tilde g)=2\,m(g).
$
\item[$\circ$]  	$
|\tilde \Sigma|_{\tilde g}=2\,|\Sigma|_{g}.
$
\item[$\circ$] $\tilde \Sigma\subset M$ is of class $C^{1,1}$.

\item[$\circ$] $\tilde \Sigma$ is area-minimizing in its homology class in $\tilde M(\tilde \Sigma)$ with respect to $\tilde g$.
\end{itemize}
\end{lem}
\begin{proof}
	The assertions follow from the above construction, using that $(M,g)$ is conformally flat at infinity, that $\partial M$ is totally geodesic at infinity,  that $\Sigma$ intersects $\partial M$ orthogonally, and that $\Sigma$ is area-minimizing in its homology class and boundary homotopy class in $M(\Sigma)$; see Figure \ref{competitor}. 	
\end{proof} 
\begin{figure}\centering
	\includegraphics[width=0.5\linewidth]{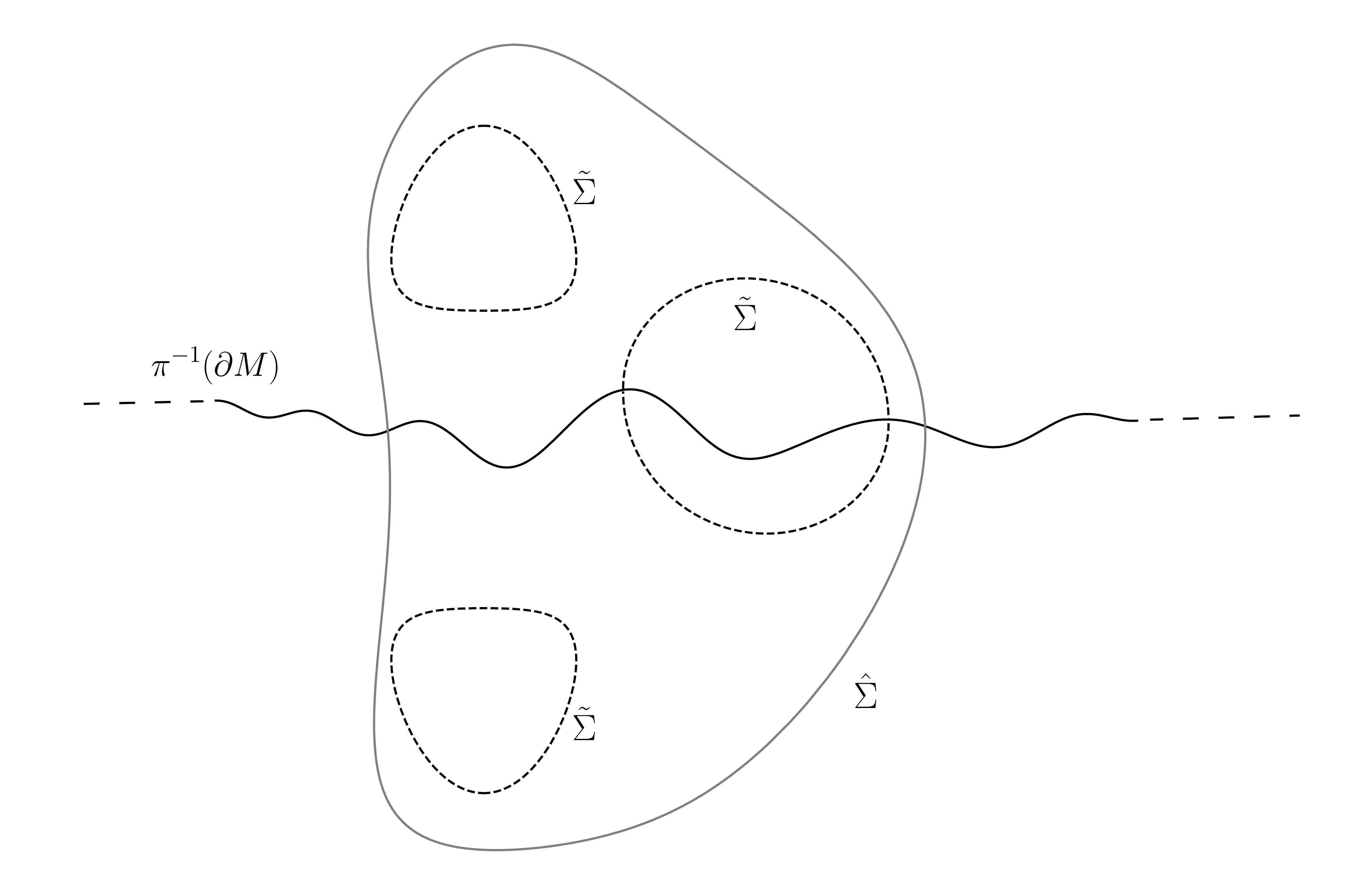}
	\caption{An illustration of the proof of Lemma \ref{mass lemma}. $\pi^{-1}(\partial M)$ is depicted by the solid black line,  $\tilde \Sigma$ is presented by the dashed line. Another hypersurface $\hat \Sigma\subset \tilde M(\tilde \Sigma)$  homologous to $\tilde \Sigma$ is presented by the solid gray line.   Each component of $\pi (\hat \Sigma\setminus \partial M)$ is homologous to $\Sigma$ and  in the same boundary homotopy class as $\Sigma$.   }
	\label{competitor}
\end{figure}
\begin{lem} \label{strictly mean convex}
	There exist a sequence $\{g^i\}_{i=1}^\infty$ of Riemannian metrics $g^i$ on $M$ and a  neighborhood $W\Subset M$ of $\Sigma$ such that $g^i$ and $g$ are conformally equivalent, $g^i=g$ outside of $W,$ and $g^i\to g$ in $C^2(M)$. Moreover,
	\begin{itemize}
	\item[$\circ$] $H(\Sigma,g^i)>0$ and
	\item[$\circ$] $H(\partial M, g^i)\geq 0$.
	\end{itemize}
\end{lem}
\begin{proof} 
		Let $\{\psi_i\}^\infty_{i=1}$ be the sequence from Lemma \ref{f conformal} and let $g^i=(1+\psi_i)^{\frac{4}{n-2}}\,g$. The assertions follow from Lemma \ref{f conformal} and Lemma \ref{conformal change}.
\end{proof} 
\begin{lem}
Notation as in Lemma \ref{strictly mean convex}. For all $\delta_0>0$ sufficiently small, the maps
$$
\Phi_i:\partial M\times[0,\delta_0)\to\{x\in M:\operatorname{dist}(x,\partial M,g^i)<\delta_0\}$$ given by 
$$\Phi_i(y,t)=\exp(g^i)_y(t\,\nu(\partial M,g^i)(y))
$$
are diffeomorphisms for all $i$. Moreover, there are smooth families $\{\gamma(g^i)_t:t\in[0,\delta_0)\}$ of Riemannian metrics $\gamma(g^i)_t$ on $\partial M$ with
$$
\Phi_i^*g^i=\gamma(g^i)_t+dt^2.
$$ 
\end{lem}
\begin{proof}
	This follows as in Lemma \ref{Phi lemma}, using also that $g^i=g$ outside of $W$ and that $g^i\to g$ in $C^2(M)$; see Lemma \ref{strictly mean convex}.
\end{proof}
As before, we consider the maps
$$
\tilde \Phi_i:\partial M\times(-\delta_0,\delta_0)\to\{\tilde x\in \tilde M:\operatorname{dist}(\pi(\tilde x),\partial M,g^i)<\delta_0\}$$ given by
$$\qquad \tilde \Phi_i(y,t)=[(\Phi_i(y,|t|),\operatorname{sign}(t))].
$$
Since $g$ and $g^i$ are conformally equivalent, the maps $\tilde \Phi_i$ are of class $C^2$. As before, 
given $t\in(-\delta_0,0]$, we define $\gamma(g^i)_{t}=\gamma(g^i)_{-t}$ and obtain a continuous metric $\tilde g^i$ given by 
$$\tilde g^i=\pi^*g^i.$$
\indent  
To smooth the metrics  $\tilde g^i$, we recall some steps from the construction in \cite[\S3]{Miao}.
To this end, let $\varphi\in C^\infty(\mathbb{R})$ with 
\begin{align*} 
&\circ \qquad\operatorname{spt}(\varphi)\subset(0,1),\qquad\qquad\qquad\qquad\qquad\qquad\qquad\qquad\qquad\qquad\qquad\qquad\qquad\qquad\quad \\&\circ \qquad 
0\leq\varphi\leq 1, \text{ and}\qquad\qquad\qquad\qquad\qquad \\&\circ \qquad  
\int_{0}^1\varphi(t)\,\mathrm{d}t=1.
\end{align*} 
Moreover, let $\eta\in C^\infty(\mathbb{R})$ with 
\begin{itemize}
	\item[$\circ$] $\operatorname{spt}(\eta)\subset(-1/2,1/2)$,
	\item[$\circ$] $\eta(t)=1/100$ if $|t|<1/4$, and
	\item[$\circ$] $0\leq \eta(t)\leq 1/100$ if $1/4<|t|<1/2.$
\end{itemize}
Let $\delta\in(0,\delta_0)$. We define $\eta_{\delta}\in C^\infty(\mathbb{R})$ by 
$$
\eta_{\delta}(t)=\delta^{2}\,\eta(\delta^{-1}\,t).
$$ Given an integer $i\geq 1$ and $t\in(-\delta_0,\delta_0)$, we define the Riemannian metric
\begin{align} \label{gamma delta} 
\gamma(\tilde g^i)^\delta_t=\int_{0}^1\gamma(\tilde g^i)_{t-t\,\eta_\delta(s)}\,\varphi(s)\,\mathrm{d}s
\end{align} 
on $\partial M$. 
\begin{lem}[{\cite[Lemma 3.2]{Miao}}]
The metric $\gamma(\tilde g^i)^\delta_t+dt^2$ is $C^2$ in $\partial M\times(-\delta_0,\delta_0)$ and agrees with $\gamma(\tilde g^i)_t+dt^2$ outside of   $\partial M\times (-\delta/2,\delta/2)$.
\end{lem}
We obtain a Riemannian metric $\tilde g^i_\delta$ on $\tilde M$ of class $C^2$ given by
\begin{equation} \label{tilde g delta} 
\tilde g^i_\delta(\tilde x)=\begin{dcases}& \tilde g^i(\tilde x) \qquad\qquad\qquad\qquad\,\text{if }  \operatorname{dist}(\pi(\tilde x),\partial M,\tilde g^i)\geq \delta_0, 
\\
&  (\tilde\Phi_i)_*(\gamma(\tilde g^i)^\delta_t+dt^2) (\tilde x) \qquad\, \text{else}.
\end{dcases}
\end{equation}
\indent The following lemma is obtained by direct computation using Lemma \ref{strictly mean convex}; cp.~\cite[pp.~1168-1170]{Miao}.
For the statement, we choose a smooth reference metric  $\check g$ on $\tilde M$ that agrees with $\tilde g$ outside of a compact set.
\begin{lem}
	There holds \label{derivative estimate} 
	$$
\limsup_{\delta\searrow0}\,\sup_{i\geq1} \big(	\delta^{-1}\,|\tilde g^i_\delta-\tilde g^i|_{\check g}+| D(\check g)\tilde g^i_\delta|_{\check g}+\delta\,| D^2(\check g)\tilde g^i_\delta|_{\check g}\big)<\infty.
	$$
	\end{lem}
In the next lemma, the assumption that $H(\partial M,g^i) \geq 0$ is used. 
\begin{lem}[{\cite[Proposition 3.1]{Miao}}]\label{scalar curvature estimate from below} 
There holds,  as $\delta\searrow 0$ and uniformly for all $i$, $R(\tilde g_\delta^i)\geq -O(1)$. 
\end{lem}
\begin{lem}
Let $i\geq 1$. There holds $H(\tilde\Sigma,\tilde g_\delta^i)>0$ on $\tilde\Sigma\setminus\pi^{-1}(\partial \Sigma)$ provided that  $\delta>0$ is sufficiently small. \label{strictly mean convex delta} 
\end{lem} 
\begin{proof}
	Without loss of generality, we may  assume that $\Sigma$ is a connected free-boundary hypersurface. It follows that $\tilde \Sigma\subset \tilde M$ is a connected, compact  hypersurface without boundary of class $C^{1,1}$ that is smooth away from $\pi^{-1}(\partial \Sigma)$. \\ \indent Let $ x_0\in \partial \Sigma$. Given $0<\delta<\delta_0$, let $U_\delta=\{\tilde x\in \tilde M:\operatorname{dist}(\pi(\tilde x), x_0, g^i)<\delta\}$.  We choose normal coordinates for $(\partial M,g^i|_{\partial M})$  centered at $x_0$ with induced frame $\partial_1,\,\dots,\partial_{n-1}$ such that $\partial_{n-1}=\nu(\Sigma,g^i)$ at $x_0$. The following error estimates are independent of the choice of $x_0$. Note that
	$$
	\tilde g^i=\sum_{\ell=1}^{n-1}d_\ell^2+dt^2+O(\delta)\qquad\text{and}\qquad \nu(\tilde \Sigma,\tilde g^i)=\partial_{n-1}+O(\delta)
	$$
in $U_\delta$.	 Using Lemma \ref{derivative estimate}, we obtain
	$$
	\tilde g^i_\delta=\sum_{\ell=1}^{n-1}d_\ell^2+dt^2+O(\delta)\qquad\text{and}\qquad \nu(\tilde \Sigma,\tilde g^i_\delta)=\partial_{n-1}+O(\delta).
	$$
Using Lemma \ref{derivative estimate} again, we conclude that, on $U_\delta\cap \tilde \Sigma \setminus \pi^{-1}(\partial \Sigma)$,
	$$
	H(\tilde \Sigma,\tilde g^i_\delta)=H(\tilde \Sigma, \tilde g^i)+\Gamma(\tilde g^i_\delta)^{n-1}_{tt}+\sum_{\ell=1}^{n-2}\Gamma(\tilde g^i_\delta)^{n-1}_{\ell\ell}-\Gamma(\tilde g^i)^{n-1}_{tt}-\sum_{\ell=1}^{n-2}\Gamma(\tilde g^i)^{n-1}_{\ell\ell}+O(\delta).
	$$
	Here, $\Gamma$ denotes a Christoffel symbol.
Using Lemma \ref{derivative estimate} once more, we have
	\begin{align*} 
2\,\Gamma(\tilde g^i_\delta)^{n-1}_{\ell\ell}=\,&2\, (\partial_\ell\,\tilde g^i_{\delta})(\partial_{n-1},\partial_\ell)-(\partial_{n-1}\,\tilde g^i_{\delta})(\partial_\ell,\partial_\ell)+O(\delta), \\ 	2\,\Gamma(\tilde g^i)^{n-1}_{\ell\ell}=\,&2\,(\partial_\ell\,\tilde g^i)(\partial_{n-1},\partial_\ell)-(\partial_{n-1}\,\tilde g^i)(\partial_\ell,\partial_\ell)+O(\delta),
	\end{align*} 
for all $1\leq \ell\leq n-2$.	Using also  \eqref{gamma delta} and \eqref{tilde g delta}, we see that
		\begin{align*} 
2\,	\Gamma(\tilde g^i_\delta)^{n-1}_{tt}=\,&2\,(\partial_t\,\tilde g^i_{\delta})(\partial_{n-1},\partial_t)-(\partial_{n-1}\,\tilde g^i_{\delta})(\partial_t,\partial_t)+O(\delta)=O(\delta), \\ 	2\,\Gamma(\tilde g^i)^{n-1}_{tt}=\,&2\,(\partial_t\,\tilde g^i)(\partial_{n-1},\partial_t)-(\partial_{n-1}\,\tilde g^i)(\partial_t,\partial_t)+O(\delta)=O(\delta) .
	\end{align*} 
Moreover, using \eqref{gamma delta} and \eqref{tilde g delta}, the same argument that led to Lemma \ref{derivative estimate} shows that
	$$
(\partial_\ell\,\tilde g^i_{\delta})(\partial_{n-1},\partial_\ell)=(\partial_\ell\,\tilde g^i)(\partial_{n-1},\partial_\ell)+O(\delta)\qquad\text{and}\qquad(\partial_{n-1}\,\tilde g^i_{\delta})(\partial_\ell,\partial_\ell)=(\partial_{n-1}\,\tilde g^i)(\partial_\ell,\partial_\ell)+O(\delta)
	$$
	for all $1\leq \ell\leq n-2$. \\ \indent 
	Since $	H(\tilde \Sigma,\tilde g^i_\delta)=H(\tilde \Sigma, \tilde g^i)$ outside of $\{\tilde x\in \tilde \Sigma:\operatorname{dist}(\pi(\tilde x), \partial \Sigma, g^i)<\delta\}$, the assertion follows.
\end{proof} 
\begin{lem} \label{mean curvature flow}
Let $i\geq 1$. 	For every $\delta>0$ sufficiently small, there exists a sequence $\{\tilde \Sigma^i_{\delta,j}\}_{j=1}^\infty$ of closed hypersurfaces $\tilde \Sigma^i_{\delta,j}\subset M\setminus M(\Sigma)$ of class $C^2$ with 
\begin{itemize}
	\item[$\circ$] $H(\tilde \Sigma^i_{\delta,j}, \tilde g_\delta^{i})>0$ for every $j$ and
	\item[$\circ$] 
	$\tilde \Sigma^\delta_{i,j}\to \tilde \Sigma$ in $C^1$.  
\end{itemize}
\end{lem}
\begin{proof}
	This follows by approximation using mean curvature flow as in the proof of \cite[Lemma 5.6]{HI} using also Lemma \ref{strictly mean convex delta}.
\end{proof} 
Recall the open set $W\Subset M$ from Lemma \ref{strictly mean convex}.
We choose open sets $W_1\Subset W_2\Subset M$ such that
\begin{itemize}
	\item[$\circ$] $W\Subset W_1$ and
	\item[$\circ$]  $\tilde g$ is $C^2$ in $M\setminus W_1$.
\end{itemize}
 We then choose a function $\chi\in C^\infty(M)$ with
\begin{itemize}
\item[$\circ$]  $0\leq\chi\leq 1$,
\item[$\circ$]  $\chi=1$ in $M\setminus W_2$, and
\item[$\circ$] $\chi=0$ in $W_1$.
\end{itemize}
We define the Riemannian metric $\hat g^i_\delta$ on $\tilde M$  by
\begin{align*}
\hat g^i_\delta=\chi\,\tilde g^i+(1-\chi)\,\tilde g^i_\delta. 
\end{align*} 
Note that $\hat g^i_\delta(\tilde x)=\tilde g^i(\tilde x)$ for all $\tilde x\in \tilde M$ with $\operatorname{dist}(\tilde x,\pi^{-1}(\partial M),\tilde g^i)\geq \delta$. 
\begin{lem} \label{hat g vs tilde g}
	There holds, as $\delta\searrow 0$ and uniformly in $i$, $$|\hat g^i_{\delta}- \tilde g^i|_{\tilde g^i}=o(1).$$ Moreover, outside of a compact subset of $\tilde M$, $\hat g^i_\delta=\tilde g^i=\tilde g$ for all $i$.
\end{lem}
\begin{proof}
	This follows from the construction using Lemma \ref{strictly mean convex} and Lemma \ref{derivative estimate}.
\end{proof} 
\begin{lem} \label{sc curvature comp}
There holds, as $\delta\searrow 0$ and uniformly in $i$, 
$$
\int_{\tilde M(\tilde \Sigma)}(\max\{-R(\hat g^i_\delta),0\})^\frac{n}{2}\,\mathrm{d}v(\hat g^i_\delta)=  2\,\int_{M(\Sigma)}(\max\{-R(g^i),0\})^\frac{n}{2}\,\mathrm{d}v(g^i)+o(1).
$$
\end{lem}
\begin{proof}
	On the one hand, if $\tilde x\in \tilde M\setminus W_2$ or $\operatorname{dist}(\tilde x,\pi^{-1}(\partial M),\tilde g^i)\geq \delta$, we have $\hat g^{i}_\delta(\tilde x)=\tilde g^i(\tilde x)$. Consequently,  
	$$
	R(\hat g^{i}_\delta)(\tilde x)=R(\tilde g^i)(\tilde x)=R(g^i)(\pi(\tilde x)).
	$$
	On the other hand, note that, as $\delta\searrow 0$,
	$$
	|\{\tilde x\in W_2:\operatorname{dist}(\tilde x,\pi^{-1}(\partial M),\tilde g^i)< \delta\}|_{\hat g^i_{\delta}}=o(1).
	$$
	Moreover, recall that, in local coordinates,
	$$
R(\hat g^i_\delta )=\sum_{a,b,k=1}^n (\hat g^i_\delta)^{ab}\bigg(\partial_k\Gamma^k_{ab}(\hat g^i_\delta )-\partial_a\Gamma^k_{kb}(\hat g^i_\delta )+\sum_{\ell=1}^n\big[\Gamma^k_{k\ell}(\hat g^i_\delta )\,\Gamma^\ell_{ab}(\hat g^i_\delta )-\Gamma^k_{a\ell}(\hat g^i_\delta )\,\Gamma^\ell_{kb}(\hat g^i_\delta )\big]\bigg).
	$$
Using Lemma \ref{derivative estimate}, Lemma \ref{scalar curvature estimate from below}, and Lemma \ref{hat g vs tilde g}, we conclude that, as $\delta\searrow0$,
$$
R(\hat g^i_\delta )=\chi\,R(\tilde g^i)+(1-\chi)\,R(\tilde g^i_\delta)+O(1)\geq -O(1)
$$
uniformly in $\{\tilde x\in W_2:\operatorname{dist}(\tilde x,\pi^{-1}(\partial M),\tilde g^i)< \delta\}$.
\\ \indent The assertion follows from these estimates. 
\end{proof} 
\begin{lem} Let $i\geq1$. For every $\delta>0$ sufficiently small, $(\tilde M,\hat g^i_{\delta})$ has horizon boundary $\tilde \Sigma^i_{\delta}\subset  \tilde M(\tilde \Sigma)$ homologous to $\tilde \Sigma$. \label{horizon boundary}
\end{lem}
\begin{proof}
	Let $\delta>0$ be sufficiently small  such that there is an integer $j_0\geq 1$ with $H(\tilde \Sigma^i_{\delta,j}, \tilde g^\delta_{i})>0$ for every $j\geq j_0$; see Lemma \ref{mean curvature flow}. It follows that there is an outermost closed minimal hypersurface $\tilde \Sigma^i_\delta\subset \tilde M(\tilde \Sigma^i_{\delta,j_0})$ homologous to $\tilde \Sigma^i_{\delta,j_0}$.  Since $H(\tilde \Sigma^i_{\delta,j}, \tilde g^\delta_{i})>0$ for every $j\geq j_0$, by the maximum principle,  $\tilde\Sigma^i_\delta$ cannot touch $\tilde \Sigma^i_{\delta,j}$ for any $j\geq j_0$. Using Lemma \ref{mean curvature flow}, it follows that $\tilde \Sigma^i_{\delta}\subset  \tilde M(\tilde \Sigma)$. 
\end{proof} 

\begin{prop} \label{gluing prop} 
There exists a sequence $\{\tilde g_i\}_{i=1}^\infty$ of $C^2$-asymptotically flat metrics of class $C^2$ on $\tilde M$ with horizon boundary $\tilde \Sigma_i\subset  M(\tilde \Sigma)$ such that
\begin{itemize}
	\item[$\circ$] $\tilde m(\tilde g_i)=2\,m(g)$,
	\item[$\circ$] $|\tilde \Sigma_i|_{\tilde g_i}=2\,|\Sigma|_g+o(1)$ as $i\to\infty$,
	\item[$\circ$] $\tilde g_i=\tilde g$ outside of a compact set, and
	\item[$\circ$] $\tilde g_i\to \tilde g$ in $C^0(\tilde M)$ as $i\to\infty$.
\end{itemize}
Moreover, there holds, as $i\to\infty$,
$$
\int_{\tilde M(\tilde \Sigma)}(\max\{-R(\tilde g_i),0\})^\frac{n}{2}\,\mathrm{d}v(\tilde g_i)= 2\,\int_{M(\Sigma)}(\max\{-R(g),0\})^\frac{n}{2}\,\mathrm{d}v( g)+o(1).
$$
For every $\alpha\in(0,1)$,  $\tilde \Sigma_i\to\tilde \Sigma$ in $C^{1,\alpha}$ with multiplicity $1$.
\end{prop}
\begin{proof}
	By Lemma \ref{strictly mean convex},
	$$
	\int_{M(\Sigma)}(\max\{-R(g^i),0\})^\frac{n}{2}\,\mathrm{d}v(g^i)=\int_{M(\Sigma)}(\max\{-R(g),0\})^\frac{n}{2}\,\mathrm{d}v(g)+o(1).
	$$
Using  Lemma \ref{sc curvature comp},	Lemma \ref{derivative estimate},  and Lemma \ref{hat g vs tilde g},   we see that, passing to a diagonal subsequence,  there is a sequence $\{\delta_i\}_{i=1}^\infty$ with $\delta_i\searrow 0$  such that the metrics $\tilde g_i=\hat g^i_{\delta_i}$ satisfy $\tilde g_i=\tilde g$ outside of a compact set,  $|\tilde g_i- \tilde g|_{\tilde g}=o(1)$,
and
$$
\int_{\tilde M(\tilde \Sigma)}(\max\{-R(\tilde g_i),0\})^{\frac{n}{2}}\,\mathrm{d}v(\tilde g_i)= 2\,\int_{M(\Sigma)}(\max\{-R(g),0\})^\frac{n}{2}\,\mathrm{d}v(g)+o(1).
$$	
Moreover, by Lemma \ref{mass lemma} and Lemma \ref{hat g vs tilde g}, $\tilde m(\tilde g_i)=2\,m(g)$.\\ 
\indent By Lemma \ref{horizon boundary}, $(\tilde M,\tilde g_i)$ has horizon boundary $\tilde \Sigma_i=\tilde \Sigma^i_{\delta_i}\subset \tilde M(\tilde \Sigma)$. By comparison with a large coordinate hemisphere, we see that
\begin{align*} 
\limsup_{i\to\infty}|\tilde \Sigma_i|_{\tilde g_i}<\infty.
\end{align*}
Moreover, we have $\pi^{-1}(\pi(\tilde \Sigma_i))=\tilde \Sigma_i$. In fact, by area-minimization, there is a closed  embedded minimal hypersurface that encloses $\pi^{-1}(\pi(\tilde \Sigma_i))$. Since $\tilde \Sigma_i$ is outermost, this minimal surface coincides with $\tilde \Sigma_i$.
\\ \indent 
  Fix $\alpha\in(0,1)$. By \cite[Regularity Theorem 1.3 (ii)]{HI}, Lemma \ref{derivative estimate}, and compactness, it follows that, passing to another subsequence if necessary, $\tilde \Sigma_i$ converges to an embedded hypersurface $\tilde\Sigma_0\subset \tilde M(\tilde \Sigma)$ of class $C^{1,\alpha}$ in $C^{1,\alpha}$ possibly with multiplicity. By standard elliptic estimates, this convergence is smooth away from $\pi^{-1}(\partial M)$ and there holds $H(\tilde \Sigma_0,\tilde g)=0$ on $\tilde\Sigma_0\setminus \pi^{-1}(\partial M)$. Since $\pi^{-1}(\pi(\tilde \Sigma_i))=\tilde \Sigma_i$, there holds $\tilde \Sigma_i\setminus \pi^{-1}(\partial M)=\tilde \Sigma_i^+\cup\tilde \Sigma_i^-$ with $\pi(\tilde \Sigma_i^\pm)=\pi(\tilde \Sigma_i)$ and $\tilde \Sigma_i^+\cap\tilde \Sigma^-_i=\emptyset$. Let $\Sigma_i=\pi(\tilde \Sigma_i^+)$. By Lemma \ref{horizon boundary}, $\Sigma_i$  is homologous to $\Sigma$ in $M(\Sigma)$ and $\partial \Sigma_i$ is  homotopy equivalent to $\partial \Sigma$ in $ M(\Sigma)\cap\partial M$. 
 Since $\Sigma$ is area minimizing in $M(\Sigma)$, we have, using also Lemma \ref{mass lemma},
  \begin{align} \label{liminf} 
  \liminf_{i\to\infty}|\tilde \Sigma_i|_{\tilde g_i}\geq 2\,\liminf_{i\to\infty}|\Sigma_i|_g\geq 2\,|\Sigma|_{g}=|\tilde \Sigma|_{\tilde g}. 
  \end{align} 
 Passing to a further subsequence if necessary, we have $\Sigma_i\to\Sigma_0\subset M(\Sigma)$  in $C^{1,\alpha}$ possibly with multiplicity, where $\Sigma_0$ satisfies $H(\Sigma_0,g)=0$ on $\Sigma_0\setminus \partial \Sigma_0$. Since $M(\Sigma)$ is an exterior region, it follows that $\Sigma_0=\Sigma$ and $\tilde \Sigma_0=\tilde \Sigma$. Since $\tilde \Sigma_i$ is area-minimizing in $M(\tilde \Sigma_i)$, we conclude that 
\begin{align*} 
\limsup_{i\to\infty}|\tilde \Sigma_i|_{\tilde g_i}\leq |\tilde \Sigma|_{\tilde g}.
\end{align*} 
In particular, using also \eqref{liminf},  we see that $\tilde \Sigma_i$ converges to $\tilde\Sigma$ with multiplicity one. \\ \indent
The assertion follows. 
\end{proof}

\section{Conformal transformation to non-negative scalar curvature}
In this section, we assume that $\tilde M$ is a smooth manifold of dimension  $3\leq n\leq 7$ and that $\tilde g$ is a Riemannian metric on $M$ of class $C^0$. We also assume that $\tilde g$ is $C^2$-asymptotically flat and that there is a closed separating  hypersurface $\tilde \Sigma\subset \tilde M$ of class $C^{1,1}$. Moreover, we assume that $\{\tilde g_i\}_{i=1}^\infty$ is a sequence of Riemannian metrics $\tilde g_i$ on $\tilde M$ of class $C^2$ with  the following properties.
\begin{itemize}
	\item[$\circ$]  $(\tilde M,\tilde g_i)$ is $C^2$-asymptotically flat with horizon boundary $\tilde \Sigma_i\subset\tilde M(\tilde \Sigma)$.
	\item[$\circ$] $\tilde m(\tilde g_i)=\tilde m(\tilde g)+o(1)$ as $i\to\infty$.
	\item[$\circ$] $|\tilde \Sigma_i|_{\tilde g_i}=|\tilde \Sigma|_{\tilde g}+o(1)$ as $i\to\infty$.
	\item[$\circ$] $\tilde g_i\to \tilde g$ in $C^0(M)$ as $i\to\infty$.
		\item[$\circ$] For every $\alpha\in(0,1)$, $\tilde \Sigma_i\to\tilde \Sigma$ in $C^{1,\alpha}$ as $i\to\infty$.
\end{itemize}
Finally, we assume that, as $i\to\infty$,
\begin{align} \label{small scalar curvature}  
\int_{\tilde M(\tilde\Sigma_i)}(\max\{0,-R(\tilde g_i)\})^\frac{n}{2}\,\mathrm{d}v(\tilde g_i)=o(1)
\end{align} 
and that, for some $\tau>(n-2)/2$,
$$
\sup_{i\geq1}\,\limsup_{\tilde x\to\infty} \left[|\tilde x|_{\bar g}^{\tau}\,|\tilde g_i-\bar g|_{\bar g}+|\tilde x|^{\tau+1}_{\bar g}\,|D(\bar g)\tilde g_i|_{\bar g}+|\tilde x|^{\tau+2}_{\bar g}\,|D^2(\bar g)\tilde g_i|_{\bar g}\right] <\infty.
$$
\indent In this section, we construct Riemannian metrics $\hat g_i$ conformally related to $\tilde g_i$ which have non-negative scalar curvature. 
\begin{prop} \label{conf change prop} 
For every $i$ sufficiently large, there exists a function $\tilde u_i\in C^\infty(\tilde M(\tilde \Sigma_i))$ such that the Riemannian metric $$\hat g_i=\tilde u_i^{\frac{4}{n-2}}\,\tilde g_i$$ has the following properties.
\begin{itemize}
\item[$\circ$] $R(\hat g_i)\geq 0$ in $\tilde M(\tilde \Sigma_i)$.
\item[$\circ$] $H(\tilde \Sigma_i,\hat g_i)=0$.
\item[$\circ$] $(\tilde M,\hat g_i)$ is $C^2$-asymptotically flat with $\tilde m(\hat g_i)=\tilde m(\tilde g)+o(1)$ as $i\to\infty$.
\item[$\circ$] $\hat g_i\to \tilde g$ in $C^0(\tilde M(\tilde \Sigma))$ as $i\to\infty$.
\item[$\circ$] For every $\alpha\in(0,1),$ $\tilde u_i\to 1$ in $C_{loc}^{1,\alpha}(\tilde M(\tilde \Sigma))$ as $i\to\infty$.
\end{itemize} 
\end{prop}

\begin{proof}
This has been proved in \cite[\S4.1]{Miao} in the special case where $\tilde\Sigma_i=\tilde \Sigma=\emptyset$. Compared to \cite[(35)]{Miao}, we take $\tilde u_i\in C^\infty(\tilde M(\tilde \Sigma))$ to be the unique solution of
\begin{equation} \label{tilde ui}
\begin{dcases} 
\qquad -\frac{4\,(n-1)}{n-2} \Delta(\tilde g_i)\,\tilde u_i-\max\{-R(\tilde g_i),0\}\,\tilde u_i=0\qquad&\text{in }\tilde M(\tilde \Sigma_i),\qquad\qquad\qquad\qquad\\\qquad 
D(\tilde g_i)_{\nu(\tilde \Sigma_i,\tilde g_i)}\tilde u_i=0 &\text{on }\tilde \Sigma_i,\text{ and}\\\qquad 
\lim_{\tilde x\to\infty} \tilde u_i=1.
\end{dcases}
\end{equation}
As shown in \cite[Lemma 3.2]{SchoenYau}, the existence of such a solution follows from \eqref{small scalar curvature} and the fact that, for every $\alpha\in(0,1)$, $\tilde \Sigma_i$ is of class $C^{1,\alpha}$. We may now repeat the proofs of \cite[Lemma 4.1, Proposition 4.1, and Lemma 4.2]{Miao};  the only difference is that the elliptic estimates for the function $\tilde u_i$ now also depend on estimates on the $C^{1,\alpha}$-regularity of $\tilde \Sigma_i$, which, by assumption, are uniform in $i$. By Lemma \ref{conformal change}, we have that $H(\tilde \Sigma_i,\hat g_i)=0$.
\end{proof}
\begin{lem} \label{horizon boundary hat }
$(\tilde M,\hat g_i)$ has horizon boundary $\hat\Sigma_i\subset \tilde M(\tilde \Sigma_i)$ and there holds
$
|\hat \Sigma_i|_{\hat g_i}\geq |\tilde \Sigma_i|_{\tilde g_i}-o(1)
$	as $i\to\infty$.
\end{lem}
\begin{proof}
Using  $H(\tilde \Sigma_i,\hat g_i)=0$, it follows that $(\tilde M,\hat g^i)$ has horizon boundary $\hat\Sigma_i\subset \tilde M(\tilde \Sigma_i)$. The assertion now follows from Proposition \ref{conf change prop}, using that $\tilde \Sigma_i$ is area-minimizing in $\tilde M(\tilde \Sigma_i)$ with respect to $\tilde g_i$. 
\end{proof}
\begin{lem}
		Suppose that there is a map $\tilde F:\tilde M\to \tilde M$ with the following properties. 
	\begin{itemize}
		\item[$\circ$] $\tilde F$ is an isometry with respect to $\tilde g_i$ for every $i$.
		\item[$\circ$] $\tilde S=\{\tilde x\in \tilde M:\tilde F(\tilde x)=\tilde x\}$ is a  separating  hypersurface such that $\tilde M\setminus \tilde S=M^+\cup M^-$ with $M^+\cap M^-=\emptyset$.
		\item[$\circ$] $\tilde g$ is $C^2$ in $M^\pm$ and  $M^\pm\cap\tilde \Sigma\subset M^\pm$ is an  outermost  minimal hypersurface. 
	\end{itemize}
Then \label{reverse area lemma}
$
|\hat \Sigma_i|_{\hat g_i}\leq |\tilde \Sigma_i|_{\tilde g_i}+o(1). 
$	
\end{lem}
\begin{proof}
	This follows as in the proof of Proposition \ref{gluing prop}.
\end{proof}
\begin{rema}
	In the situation of Section \ref{gluing section}, we may take $\tilde F:\tilde M\to\tilde M$ to be the unique map with $\pi\circ\tilde F=\pi$ and $\tilde F\neq \operatorname{Id}$.
\end{rema}
\section{Proof of Theorem \ref{main thm}}
Let $(M,g)$ be an asymptotically flat half-space of dimension $3\leq n\leq 7$ with horizon boundary $\Sigma\subset M$ and such that $R(g)\geq 0$ in $M(\Sigma)$ and $H(\partial M,g)\geq 0$ on $ M(\Sigma)\cap\partial M$.
\begin{proof}[Proof of Theorem \ref{main thm}]
Using Proposition \ref{ahf to conformally flat prop} to obtain a conformally flat approximation of $(M,g)$, Proposition \ref{gluing prop} to double the approximation, and Proposition \ref{conf change prop}, Lemma \ref{horizon boundary hat }, and Lemma \ref{reverse area lemma} to conformally transform the double to non-negative scalar curvature,  we see that there exists a smooth manifold $\tilde M$ of dimension $n$ and a sequence $\{\hat g_i\}_{i=1}^\infty$ of Riemannian metrics $\hat g_i$ on $\tilde M$  with  the following properties.
\begin{itemize}
	\item[$\circ$]  $(\tilde M,\hat g_i)$ is asymptotically flat with horizon boundary $\hat \Sigma_i\subset \tilde M$
	\item[$\circ$] $R(\hat g_i)\geq 0$ in $\tilde M(\hat \Sigma_i)$
	\item[$\circ$] $\tilde m(\hat g_i)=2\,m(g)+o(1)$ as $i\to\infty$
	\item[$\circ$] $|\hat \Sigma_i|_{\hat g_i}=2\,|\Sigma|_{g}+o(1)$ as $i\to\infty$
\end{itemize}

\indent The assertion follows.
\end{proof}
\section{Mass-decreasing variations and rigidity} 
Let $(M,g)$ be an asymptotically flat half-space of dimension $3\leq n\leq7$ with horizon boundary $\Sigma\subset M$ such that $R(g)\geq 0$ in $M(\Sigma)$ and $H(\partial M,g)\geq 0$ on $ M(\Sigma)\cap\partial M$. We also assume that equality holds in \eqref{coro RPI}, i.e.\,, that
$$
m(g)=\left(\frac{1}{2}\right)^{\frac{n}{n-1}}\,\left(\frac{|\Sigma|_g}{\omega_{n-1}}\right)^{\frac{n-2}{n-1}}.
$$
\indent The goal of this section is to show that $(M(\Sigma),g)$ is isometric to the  exterior region of a Schwarzschild half-space \eqref{Schwarzschild half space}. \\ \indent
The argument in Lemma \ref{scalar flat} below is modeled on the proof of \cite[Corollary 3.1]{SchoenYau}, where R.~Schoen and S.-T.~Yau study the equality case of the positive mass theorem.
\begin{lem}[Cp.~{\cite[Corollary 3.1]{SchoenYau}}]
	There holds $R(g)=0$ in $M(\Sigma)$. \label{scalar flat}
\end{lem}
\begin{proof}
	Suppose, for a contradiction, that $R(g)\neq 0$. Recall the definition \eqref{weighted hoelder} of the weighted Hölder space $C^{2,\alpha}_{\tau}(M(\Sigma))$. By Proposition \ref{PDE prop 2}, there is  a unique solution $v\in C_\tau^{2,\alpha}(M(\Sigma))$ of 
	\begin{equation*}
	\begin{dcases} \qquad
	-\frac{4\,(n-1)}{(n-2)}\Delta_{g}\,v+R(g)\,(1+v)=0\qquad &\text{in }\operatorname{int}(M(\Sigma)) ,\qquad\qquad\qquad\qquad\\\qquad
	D(g)_{\nu(\partial M,g)}v=0&\text{on } M(\Sigma)\cap\partial M,\text{ and} \\\qquad
	v=0& \text{on } \Sigma.
	\end{dcases} 
	\end{equation*} Since $R(g)\neq 0$, $v$ is non-constant.  By the maximum principle, we have $-1<v<0$ in $M(\Sigma)\setminus \Sigma $ and
	$
	D(g)_{\nu(\Sigma,g)}v<0
	$
	on $\Sigma$. We define the family $\{g_t\}_{t\in[0,1)}$ of Riemannian metrics $$g_t=(1+t\,v)^{\frac{4}{n-2}}\,g$$ on $M(\Sigma)$. Note that $g_t$ is asymptotically flat for every $t\in[0,1)$. Moreover, by Lemma \ref{conformal change}, we have, for every $t\in(0,1)$,
	\begin{itemize}
		\item[$\circ$] $R(g_t)\geq0$ in $M(\Sigma)$,
\item[$\circ$]  $H(\partial M,g_t)\geq 0$ on $M(\Sigma)\cap \partial M$, and
\item[$\circ$] $H(\Sigma,g_t)>0$.	
\end{itemize}
	 \indent \indent Arguing as in the proof of Proposition \ref{ahf to conformally flat prop}, we find that  $(M,g_t)$ has horizon boundary $\Sigma_t\subset M(\Sigma)$ and that $\Sigma_t\to\Sigma$ smoothly as $t\searrow 0$. Using that $v=0$ on $\Sigma$ and that $H(\Sigma,g)=0$, we conclude that
	$$
	\lim_{t\searrow0}t^{-1}\,\big(|\Sigma_t|_{g_t}-|\Sigma|_g\big)=0.
	$$ \indent
	Next, we compute $m(g_t)$ in the asymptotically flat chart of $(M,g)$. By \eqref{half-space mass intro},
	\begin{equation*} 
	\begin{aligned}   
	&2\,(n-1)\,\omega_{n-1}\,m(g_t)
	\\&\qquad=\lim_{\lambda\to\infty}\lambda^{-1}\,\bigg(\sum_{i,\,j=1}^n\int_{ \mathbb{R}^n_+\cap {S}^{n-1}_\lambda(0)}(1+t\,v)^{\frac{4}{n-2}}\,x^i\,\big[(\partial_jg)(e_i,e_j)-(\partial_ig)(e_j,e_j)\big]\,\mathrm{d}\mu(\bar g)
	\\&\qquad \qquad \qquad \qquad \qquad +\sum_{i=1}^{n-1}\int_{(\mathbb{R}^{n-1}\times\{0\})\cap S^{n-1}_\lambda(0)} (1+t\,v)^{\frac{4}{n-2}}\,x^i\,g(e_i,e_n)\,\mathrm{d}l(\bar g) \bigg )
	\\&\qquad \qquad  +\frac{4}{n-2}\,t\,\lim_{\lambda\to\infty}\lambda^{-1}\,\sum_{i,\,j=1}^n\int_{ \mathbb{R}^n_+\cap{S}^{n-1}_\lambda(0)}(1+t\,v)^{\frac{6-n}{n-2}}\,x^i\,\big[\partial_jv\,g(e_i,e_j)-\partial_iv\,g(e_j,e_j)\big]\,\mathrm{d}\mu(\bar g).
	\end{aligned}
	\end{equation*} 
	Using that $g$ is asymptotically flat \eqref{af intro} and that $v\in C_{\tau}^{2,\alpha}(M(\Sigma))$, we have
$$\lim_{\lambda\to\infty}\lambda^{-1}\,\sum_{i,\,j=1}^n\int_{ \mathbb{R}^n_+\cap {S}^{n-1}_\lambda(0)}v\,x^i\,\big[(\partial_jg)(e_i,e_j)-(\partial_ig)(e_j,e_j)\big]\,\mathrm{d}\mu(\bar g)=0$$ and $$
		 \lim_{\lambda\to\infty}\lambda^{-1}\sum_{i=1}^{n-1}\int_{(\mathbb{R}^{n-1}\times\{0\})\cap S^{n-1}_\lambda(0)} v\,x^i\,g(e_i,e_n)\,\mathrm{d}l(\bar g) =0.
$$
It follows that
	$$
	\lim_{t\searrow 0} t^{-1}\,\big(m(g_t)-m(g)\big)=-\frac{2}{(n-2)\,\omega_{n-1}}\,\lim_{\lambda\to\infty}\lambda^{-1}\, \sum_{i=1}^n\int_{ \mathbb{R}^n_+\cap {S}^{n-1}_\lambda(0)}x^i\,\partial_iv\,\mathrm{d}\mu(\bar g).
	$$
	In conjunction with Lemma \ref{asymptotic growth}, we conclude that 
$$
	\lim_{t\searrow 0} t^{-1}\,\big(m(g_t)-m(g)\big)<0
$$
and, in particular, that
	$$
	\lim_{t\searrow 0} t^{-1}\,\bigg[m(g_t)-\left(\frac{1}{2}\right)^{\frac{n}{n-1}}\,\left(\frac{|\Sigma_t|_{g_t}}{\omega_{n-1}}\right)^{\frac{n-2}{n-1}}	-m(g)+\left(\frac{1}{2}\right)^{\frac{n}{n-1}}\,\left(\frac{|\Sigma|_{g}}{\omega_{n-1}}\right)^{\frac{n-2}{n-1}}\bigg]<0.
	$$
As this is not compatible with 	Corollary \ref{penrose coro}, the assertion follows.

\end{proof}
\begin{lem}
	There holds $H(\partial M,g)=0$ on $ M(\Sigma)\cap\partial M$. \label{mean flat} 
\end{lem}
\begin{proof}
	Suppose, for a contradiction, that there is $U\Subset M(\Sigma)\setminus \Sigma$ open with $U\cap \partial M\neq \emptyset$ such that  $H(\partial M,g)>0$ on $U\cap \partial M$.  \\ \indent 
	By Lemma \ref{scalar curvature for mean curvature}, there exists  $\psi\in C^\infty(M)$ with the following properties.
	 \begin{itemize}
			\item[$\circ$] $\psi$ has compact support in $U\cap \partial M $.
		\item[$\circ$] $\Delta_g\psi \geq 0$ in $M(\Sigma)$ and $\Delta_g\psi> 0$ at some point.

	\end{itemize}
	Let $0<t_0<\sup_{x\in M}|\psi|$. We define the family $\{g_t\}_{t\in[0,t_0)}$ of Riemannian metrics
	$$
	g_t=(1-t\,\psi)^{\frac{4}{n-2}}\,g.
	$$ 
	By Lemma \ref{conformal change} and Lemma \ref{scalar flat}, there holds $R(g_t)\geq 0$ in $M(\Sigma)$ and $R(g_t)> 0$ at some point  for every $t\in(0,t_0)$. Decreasing $t_0>0$ if necessary, we also have $H(\partial M,g_t)\geq 0$ on $ M(\Sigma)\cap\partial M$ for every $t\in(0,t_0)$.\\ \indent On the one hand, arguing as in the proof of Proposition \ref{ahf to conformally flat prop}, $(M,g_t)$ has horizon boundary $\Sigma_t\subset M(\Sigma)$ with $\Sigma_t\to \Sigma$ smoothly as $t\searrow 0$. Using that $g_t=g$ near $\Sigma$ and that $M(\Sigma)$ is an exterior region, we conclude that $\Sigma_t=\Sigma$ for every $t\in(0,t_0)$ sufficiently small. On the other hand, clearly, $m(g_t)=m(g)$ for every $t\in(0,t_0)$. It follows that 
	$$
	m(g_t)=\left(\frac{1}{2}\right)^{\frac{n}{n-1}}\,\left(\frac{|\Sigma_t|_{g_t}}{\omega_{n-1}}\right)^{\frac{n-2}{n-1}}
	$$
	for every $t\in(0,t_0)$ sufficiently small.
	As this is not compatible with Lemma \ref{scalar flat}, the assertion follows.
\end{proof} 
The following lemma is the key technical step in the proof of Theorem \ref{rigidity remark}.
\begin{lem} \label{perturbation}
	Suppose that there is   $U\Subset M(\Sigma)\setminus \Sigma $ open with $U\cap \partial M\neq\emptyset$ such that $h(\partial M,g)\neq 0$ on $U\cap \partial M$.    There exists a smooth family $\{g_t\}_{t\in[0,t_0)}$ of Riemannian metrics $g_t$ on $M$ such that
	\begin{itemize}
		\item[$\circ$] $g_t=g$ outside of $U$ and $g_0=g$, 
		\item[$\circ$] $\lim_{t\searrow 0} t^{-1}(|U\cap\partial M|_{g_t}-|U\cap\partial M|_g)=0$, and 
		\item[$\circ$] $\lim_{t\searrow 0} t^{-1}\,H(\partial M, g_t)= 0$ on $ M(\Sigma)\cap\partial M$.
	\end{itemize} 
Moreover,
\begin{equation*}
\begin{aligned}
\lim_{t\searrow0}\,t^{-1}\,R(g_t)\geq 0
\end{aligned}
\end{equation*}
with strict inequality at some point.
\end{lem}
\begin{proof}
	Let $x_0\in \partial M\setminus \Sigma$ be such that $h(\partial M,g)(x_0) \neq 0$. Using that $H(\partial M,g)=0$ on $M(\Sigma)\cap\partial M$, see Lemma \ref{mean flat}, we see that there is an orthonormal basis  	 $e_1,\,e_2,\dots,e_{n-1}$ of $T_{x_0}(\partial M)$ of principal directions of $h(\partial M,g)$ with
\begin{equation}   \label{p order} 
\begin{aligned}
&h(\partial M,g)(x_0)(e_1,e_1)\geq\max\{|h(\partial M,g)(x_0)(e_2,e_2)|,\dots,|h(\partial M,g)(x_0)(e_{n-1},e_{n-1})|\},\\
&\qquad h(\partial M,g)(x_0)(e_1,e_1)>0,\qquad\text{and}\qquad -h(\partial M,g)(x_0)(e_2,e_2)> 0.
\end{aligned} 
\end{equation} 
\indent 
Given $\varepsilon>0$ sufficiently small,	we define a local parametrization
	$$
	\Psi_\varepsilon:\{y\in\mathbb{R}^{n-1}:|y|_{\bar g}<\varepsilon\}\to \partial M \qquad\text{given by} \qquad \Psi_\varepsilon(y)=\exp(g|_{\partial M})_{x_0}\bigg(\sum_{\ell=1}^{n-1}y^\ell\,e_\ell\bigg)
	$$
	of $\partial M$ near $x_0$. Decreasing $\varepsilon>0$ if necessary, we obtain a local parametrization 
	$$
	\Phi_\varepsilon:\{y\in\mathbb{R}^{n-1}:|y|_{\bar g}<\varepsilon\}\times\{s\in\mathbb{R}:0\leq s< \varepsilon\}\to M$$ of $M$ near $x_0$ given by
	$$\Phi_\varepsilon(y,s)=\exp(g)_{\Psi_\varepsilon(y)}(s\,\nu(\partial M,g)(\Psi_\varepsilon(y))).
	$$
	By construction, 
	\begin{align} \label{chart metric} 
	\Phi_\varepsilon^*g=\gamma_s+ds^2
	\end{align} 
	where, for each $s\in[0,\varepsilon)$, $\gamma_s$ is a Riemmanian metric on $\{y\in\mathbb{R}^{n-1}:|y|_{\bar g}<\varepsilon\}$. Note that for all $s\in[0,\varepsilon]$, as $\varepsilon\searrow 0$,
	\begin{equation}\label{properties} \begin{aligned}
	&\circ\qquad\gamma_s=\bar g|_{\mathbb{R}^{n-1}}+o(1),\qquad\qquad\qquad\qquad \qquad \\
	&\circ\qquad D(\bar g)\gamma_s=O(1), \text{ and}\\
&\circ\qquad D^2(\bar g)\gamma_s=O(1).
	\end{aligned}
	\end{equation}
	Moreover, computing the Christoffel symbols of $g$ in the chart $\Phi_{\varepsilon}$, we see that, as $\varepsilon \searrow 0$,
	\begin{align} \label{h vs gamma} 
\sup_{s\in[0,\varepsilon)}\,\sup_{\,|y|_{\bar g}<\varepsilon}\bigg|\frac12\,	(D(\bar g)_{e_n}\gamma_s)(y)+h(\partial M,g)(\Psi_\varepsilon(y))\bigg|_{\bar g}=o(1).
	\end{align} 
	In particular,
	\begin{align} \label{mc gamma} 
\sup_{s\in[0,\varepsilon)}\,\sup_{\,|y|_{\bar g}<\varepsilon}\operatorname{tr}_{\bar g|_{\mathbb{R}^{n-1}}}(D(\bar g)_{e_n}\gamma_s)(y)=o(1).
\end{align}	
By \eqref{p order}, there holds on $\Phi( \mathbb{R}^n_+\cap B^n_\varepsilon(0))$ for all $1\leq i,\, j\leq n-1$ with $i\neq j$,  as $\varepsilon\searrow 0$,
\begin{equation}\label{p order 3}
\begin{aligned}  
h(\partial M,g)(e_1,e_1)\geq\,&|h(\partial M,g)(e_i,e_i)|+o(1), \\
|h(\partial M,g)(e_i,e_j)|=\,&o(1),\\
h(\partial M,g)(e_1,e_1)\geq\,&h(\partial M,g)(x_0)(e_1,e_1)-o(1), \text{ and}
\\-h(\partial M,g)(e_2,e_2)\geq\,& 0.
\end{aligned}
\end{equation} 
 \indent 
 Let $K>1$ large be a constant to be chosen later, $\rho\in C^\infty(\mathbb{R}^{n-1})$ be the function from Lemma \ref{scalar curvature balance}, and $\delta\in(0,\varepsilon)$. We define $\rho_\varepsilon\in C^\infty(\mathbb{R}^{n-1})$ by
 $\rho_\varepsilon(y)=\rho(\varepsilon^{-1}\,y)$ and 	\begin{equation} \label{eta def} 
	\eta_{\delta}:[0,\infty)\to\mathbb{R}\qquad\text{ by}\qquad \eta_\delta(s)=\begin{dcases}
	\exp^{-(1-\delta^{-1}\,s)^{-1}}\qquad\,\,\,\,&\text{if}\qquad0\leq s<\delta, 
	\\0&\text{else}.
	\end{dcases}
	\end{equation} 
	Let $$f_{\varepsilon,\delta}(y,s)=\eta_\delta(s)\,\rho_\varepsilon(y).$$ By the Gram-Schmidt process, given $(y,s)\in\mathbb{R}^{n-1}\times\mathbb{R}$ with $|y|_{\bar g}<\varepsilon$ and $0\leq s<\varepsilon$, there exists an invertible $(n-1)\times(n-1)$-matrix $A=A(y,s)$ smoothly depending on $(y,s)$ such that $\gamma_s=A^t\,A$. Moreover, as $\varepsilon\searrow 0$,
	\begin{align} \label{O estimate}
	A=\operatorname{Id}+o(1),\qquad  D(\bar g)A=O(1),\qquad \text{and}\qquad D^2(\bar g)A=O(1).
	\end{align}
Let $E:\mathbb{R}^{n-1}\to\mathbb{R}^{n-1}$ be the linear map given by $E(e_1)=-e_1$, $E(e_2)=e_2$, and $E(e_i)=0$ when $3\leq i\leq n-1$. We define a  symmetric $(0,2)$-tensor $\sigma^{\varepsilon,\delta}$ on  $\mathbb{R}^n_+\cap B^n_\varepsilon(0)$ by
\begin{equation*}
		\begin{aligned} 
&\sigma^{\varepsilon,\delta}|_{\mathbb{R}^{n-1}}=f_{\varepsilon,\delta}\,A^t\,E\,A,\\& \sigma^{\varepsilon,\delta}(e_n,e_i)=\sigma^{\varepsilon,\delta}(e_i,e_n)=0 \text{ for }\,i=1,\,2,\dots,n-1, \text{ and} \\&
\sigma^{\varepsilon,\delta}(e_n,e_n)=-K\,f_{\varepsilon,\delta}.
	\end{aligned} 
\end{equation*}
	Note that 
	\begin{align} \label{tr one} 
	\operatorname{tr}_{\Phi_\varepsilon^*g}\sigma^{\varepsilon,\delta}=-K\,f_{\varepsilon,\delta}\end{align}  and, for each $s\in[0,\varepsilon)$, \begin{align} \label{tr two} \operatorname{tr}_{\gamma_s}(\sigma^{\varepsilon,\delta}|_{\mathbb{R}^{n-1}})=0.
	\end{align}
	For $t_0>0$ sufficiently small, we obtain a family $\{g^{\varepsilon,\delta}_t\}_{t\in[0,t_0)}$ of Riemannian metrics $g^{\varepsilon,\delta}_t$ on $M$ where
	$$
	g^{\varepsilon,\delta}_t(x)=\begin{dcases}&g(x)+t\,((\Phi_\varepsilon)_*\sigma^{\varepsilon,\delta})(x)\qquad\text{if } x\in\operatorname{Im}(\Phi_{\varepsilon}),\\
	&g(x)\qquad\qquad\qquad \qquad\quad\,\,\,\,\,\text{else}.
	\end{dcases}
	$$
	Note that
	$$
	\Phi^*_\varepsilon g_t^{\varepsilon,\delta}=\gamma_s+ds^2+t\,\sigma^{\varepsilon,\delta}.
	$$
	Moreover, by \eqref{tr two},
	$$
	\lim_{t\searrow 0} t^{-1}(|U\cap\partial M|_{g^{\varepsilon,\delta}_t}-|U\cap\partial M|_g)=0.
	$$
	 \indent 
We estimate the linearization of $R(g^{\varepsilon,\delta}_t)$ at $t=0$. All geometric expressions below are computed in the chart $\Phi_\varepsilon$.  Recall from \cite[(6.7)]{kazdanwarner} that
	\begin{align} \label{kazdanwarner} 
	\lim_{t\searrow 0}\,t^{-1}\,R(g^{\varepsilon,\delta}_t)=\operatorname{div}_g\operatorname{div}_g\sigma^{\varepsilon,\delta}-\Delta_g\operatorname{tr}_g\sigma^{\varepsilon,\delta}-g(\operatorname{Ric}(g),\sigma^{\varepsilon,\delta}).
	\end{align} 
	By \eqref{tr one}, \eqref{chart metric}, and \eqref{mc gamma}, we have, as $\varepsilon\searrow 0$, 
	\begin{align*} 
	\Delta_g\operatorname{tr}_g\sigma^{\varepsilon,\delta}=\,&-K\,\eta''_\delta\,\rho_\varepsilon+K\,\eta_\delta'\,\rho_{\varepsilon}\,\sum_{i,\,j=1}^{n-1}g^{ij}\,\Gamma(g)^n_{ij}-K\,\eta_\delta\,\sum_{i,j=1}^{n-1}g^{ij}\,\left[\partial_i\partial_j\rho_{\varepsilon}-\sum_{\ell=1}^{n-1}\Gamma(g)^\ell_{ij}\,\partial_\ell\rho_{\varepsilon}\right]
	\\=\,&-K\,\eta''_\delta\,\rho_\varepsilon-K\,\eta_\delta\,\sum_{i,j=1}^{n-1}\left[g^{ij}\,\partial_i\partial_j\rho_{\varepsilon}-\sum_{\ell=1}^{n-1}g^{ij}\,\Gamma(g)^\ell_{ij}\,\partial_\ell\rho_{\varepsilon}\right]+K\,o(\eta_\delta'\,\rho_{\varepsilon}).
	\end{align*}
	Moreover, 
	$$
	g(\operatorname{Ric}(g),\sigma^{\varepsilon,\delta})=K\,O(f_{\varepsilon,\delta}).
	$$
Next, we compute
	\begin{align*}
	&\operatorname{div}_g\operatorname{div}_g{\sigma^{\varepsilon,\delta}}\\&\qquad =\sum_{a,\,b,\,i,\,j=1}^n\bigg[g^{ab}\,g^{ij}\,\partial_i\partial_a \sigma^{\varepsilon,\delta}_{jb}+g^{ab}\,\partial_a g^{ij}\,\partial_i \sigma^{\varepsilon,\delta}_{jb}\\&\qquad\qquad\qquad\quad-\sum_{\ell=1}^n\bigg(g^{ab}\,g^{ij}\,\Gamma(g)^\ell_{ij}\,\partial_a\sigma^{\varepsilon,\delta}_{\ell b}+g^{ab}\,g^{ij}\,\Gamma(g)^{\ell}_{ib}\,\partial_a\sigma^{\varepsilon,\delta}_{j\ell}+g^{ab}\,g^{ij}\,\Gamma(g)^{\ell}_{ab}\,\partial_i\sigma^{\varepsilon,\delta}_{j\ell}\bigg)\bigg]\\
	&\qquad\qquad+K\,O(f_{\varepsilon,\delta})
	\end{align*}
	Using \eqref{chart metric}, \eqref{properties}, and \eqref{O estimate},
	we have
	$$
\sum_{a,\,b,\,i,\,j=1}^n	g^{ab}\,g^{ij}\,\partial_i\partial_a \sigma^{\varepsilon,\delta}_{jb}=-K\,\eta''_\delta\,\rho_\varepsilon+O(\eta_\delta\,D^2(\bar g|_{\mathbb{R}^{n-1}})\rho_{\varepsilon})+O(\eta_\delta\,D(\bar g|_{\mathbb{R}^{n-1}})\rho_{\varepsilon})+O(f_{\varepsilon,\delta})
	$$
	and, using also \eqref{mc gamma},
	\begin{align*} 
&\sum_{a,\,b,\,i,\,j=1}^n\bigg[	g^{ab}\,\partial_a g^{ij}\,\partial_i \sigma^{\varepsilon,\delta}_{jb}-\sum_{\ell=1}^n\bigg( g^{ab}\,g^{ij}\,\Gamma(g)^\ell_{ij}\,\partial_a\sigma^{\varepsilon,\delta}_{\ell b}+g^{ab}\,g^{ij}\,\Gamma(g)^{\ell}_{ab}\,\partial_i\sigma^{\varepsilon,\delta}_{j\ell}\bigg)\bigg]\\&\qquad=2\,K\,\eta'_\delta\,\rho_{\varepsilon}\,\sum_{i,\,j=1}^{n-1}g^{ij}\,\Gamma(g)^n_{ij} +O(\eta_\delta\,D(\bar g|_{\mathbb{R}^{n-1}})\rho_{\varepsilon})+O(f_{\varepsilon,\delta}) 
\\&\qquad =K\,o(\eta'_\delta\,\rho_{\varepsilon})+O(\eta_\delta\,D(\bar g|_{\mathbb{R}^{n-1}})\rho_{\varepsilon})+O(f_{\varepsilon,\delta}). 
	\end{align*} 
	Likewise, using also \eqref{h vs gamma} and \eqref{p order 3},
	\begin{align*} 
	g^{ab}\,g^{ij}\,\Gamma(g)^{\ell}_{ib}\,\partial_a\sigma^{\varepsilon,\delta}_{j\ell}=\,&[h(\partial M,g)(e_1,e_1)-h(\partial M,g)(e_2,e_2)]\,\eta'_\delta\,\rho_\varepsilon\\&\qquad +o(\eta'_\delta\,\rho_{\varepsilon})+O(\eta_\delta\,D(\bar g|_{\mathbb{R}^{n-1}})\rho_{\varepsilon})+O(f_{\varepsilon,\delta}).
	\end{align*} 
	We conclude that, as $\varepsilon\searrow 0$, 
	\begin{align*} 
		\lim_{t\searrow 0}\,t^{-1}\,R(g^{\varepsilon,\delta}_t)&\geq -[h(\partial M,g)(e_1,e_1)-h(\partial M,g)(e_2,e_2)]\,\eta'_\delta\,\rho_\varepsilon \\&\qquad
		+K\,\eta_\delta\,\sum_{i,j=1}^{n-1}\left[g^{ij}\,\partial_i\partial_j\rho_{\varepsilon}-\sum_{\ell=1}^{n-1}g^{ij}\,\Gamma(g)^\ell_{ij}\,\partial_\ell\rho_{\varepsilon}\right]
		\\&\qquad 
		 -K\,o(\eta'_\delta\,\rho_{\varepsilon})-O(\eta_\delta\,D^2(\bar g|_{\mathbb{R}^{n-1}})\rho_{\varepsilon})-O(\eta_\delta\,D(\bar g|_{\mathbb{R}^{n-1}})\rho_{\varepsilon})-K\,O(f_{\varepsilon,\delta}).
	\end{align*} 
	By Lemma \ref{scalar curvature balance} and \eqref{properties}, we may choose $K>1$ such that, for all $y\in\mathbb{R}^{n-1}$ with $\varepsilon/2\leq |y|_{\bar g}<\varepsilon$ and $s\in[0,\delta)$,
	$$
K\,\eta_\delta\,\sum_{i,j=1}^{n-1}\left[g^{ij}\,\partial_i\partial_j\rho_{\varepsilon}-\sum_{\ell=1}^{n-1}g^{ij}\,\Gamma(g)^\ell_{ij}\,\partial_\ell\rho_{\varepsilon}\right]\geq 	O(\eta_\delta\,D^2(\bar g|_{\mathbb{R}^{n-1}})\rho_{\varepsilon})+O(\eta_\delta\,D(\bar g|_{\mathbb{R}^{n-1}})\rho_{\varepsilon})+K\,O(f_{\varepsilon,\delta}).
	$$
	 provided that $\varepsilon>0$  is sufficiently small. Moreover, by \eqref{p order 3}, we have, for all $y\in\mathbb{R}^{n-1}$ with $|y|_{\bar g}<\varepsilon$ and $s\in[0,\delta)$,
	 \begin{align} \label{strict} 
	- [h(\partial M,g)(e_1,e_1)-h(\partial M,g)(e_2,e_2)]\,\eta'_\delta\,\rho_\varepsilon-K\, o(\eta'_\delta\,\rho_{\varepsilon})>0
	 \end{align} 
	 provided that $\varepsilon>0$ is sufficiently small. Consequently, for all $y\in\mathbb{R}^{n-1}$ with $\varepsilon/2\leq |y|_{\bar g}<\varepsilon$ and $s\in[0,\delta)$,
	$$
\lim_{t\searrow 0}\,t^{-1}\,R(g^{\varepsilon,\delta}_t)\geq 0 
$$
provided that $\varepsilon>0$ is sufficiently small. 
	 Finally, by \eqref{eta def} we have $\eta_\delta=o(\eta'_\delta)$ as $\delta\searrow 0$ and, by Lemma \ref{scalar curvature balance}, we have $$\liminf_{\varepsilon\searrow 0}\sup\{\rho_{\varepsilon}(y):y\in\mathbb{R}^{n-1}\text{ and }|y|_{\bar g}\leq \varepsilon/2\}>0.$$ Using \eqref{strict}, we conclude that,  for all $y\in\mathbb{R}^{n-1}$ with $|y|_{\bar g}\leq \varepsilon/2$ and $s\in[0,\delta)$,
	$$
		\lim_{t\searrow 0}\,t^{-1}\,R(g^{\varepsilon,\delta}_t)> 0 
	$$
	 provided that $\varepsilon>0$ and $\delta\in(0,\varepsilon)$ are sufficiently small.\\
	  \indent
Next, we compute	the linearization of $H(\partial M, g^{\varepsilon,\delta}_t)$ at $t=0$. As before, all geometric expressions are computed in the chart $\Phi_\varepsilon$.
The argument that led to \eqref{h vs gamma} also shows that
	\begin{align*} 
h(\partial M,g^{\varepsilon,\delta}_t)(\Psi_\varepsilon)=-\frac12\,	\big(D(\bar g)_{e_n}\gamma_s+t\,(D(\bar g)_{e_n} \sigma^{\varepsilon,\delta})|_{\mathbb{R}^{n-1}}\big)\big|_{s=0}
\end{align*} 
and
	$$
	H(\partial M,g^{\varepsilon,\delta}_t)=-\frac12\,\operatorname{tr}_{\bar g|_{\mathbb{R}^{n-1}}}\left[(\gamma_0+t\,\sigma^{\varepsilon,\delta}|_{\mathbb{R}^{n-1}})^{-1}\,\big(D(\bar g)_{e_n}\gamma_s+t\,(D(\bar g)_{e_n}\sigma^{\varepsilon,\delta})|_{\mathbb{R}^{n-1}}\big)\big|_{s=0}\right].
	$$
Using that $\gamma_s=A^t\,A$ and that $\sigma^{\varepsilon,\delta}|_{\mathbb{R}^{n-1}}=A^T\,E\,A$,	we obtain
\begin{align*} 
&H(\partial M,g_t^{\varepsilon,\delta})\\&\qquad =-\frac12\,\operatorname{tr}_{\bar g|_{\mathbb{R}^{n-1}}}\left[(A^t\,(\operatorname{Id}+t\,f_{\varepsilon,\delta}\,E)\,A)^{-1}\,D(\bar g)_{e_n}(A^t\,(\operatorname{Id}+t\,f_{\varepsilon,\delta}\,E)\,A)\right]\big|_{s=0}
\\&\qquad =
-\frac12\,\operatorname{tr}_{\bar g|_{\mathbb{R}^{n-1}}}\left[(A^t)^{-1}\,D(\bar g)_{e_n}A^t+A^{-1}\,D(\bar g)_{e_n}A+(\operatorname{Id}+t\,f_{\varepsilon,\delta}\,E)^{-1}\,t\,(D(\bar g)_{e_n} f_{\varepsilon,\delta})\,E\right]\big|_{s=0}.
\end{align*} 
Note that
$$
\frac{d}{dt}\bigg|_{t=0}\operatorname{tr}_{\bar g|_{\mathbb{R}^{n-1}}}\left[(\operatorname{Id}+t\,f_{\varepsilon,\delta}\,E)^{-1}\,t\,\,E\right]=\operatorname{tr}_{\bar g|_{\mathbb{R}^{n-1}}}[E]=0.
$$

We conclude that
	\begin{align*} 
	\frac{d}{dt}\bigg|_{t=0} H(\partial M,g^{\varepsilon,\delta}_t)=0.
	\end{align*}
\indent The assertion follows.  
	\end{proof} 
	\begin{lem} \label{totally geodesic} 
		There holds $h(\partial M,g)=0$ on $M(\Sigma)\cap\partial M$.
	\end{lem}
\begin{proof} 
Suppose, for a contradiction, that there is  $U\Subset M(\Sigma)\setminus \Sigma$ open with $U\cap\partial M\neq \emptyset$ such that $h(\partial M,g)\neq 0$ on $U\cap \partial M$. Let $\{g_t\}_{t\in[0,t_0)}$ be the family of Riemannian metrics on $M$  from Lemma \ref{perturbation}. 
 Arguing as in the proof of \cite[Lemma 4.3]{ABDL}, using Proposition \ref{PDE prop 2} instead of \cite[Proposition 3.3]{ABDL}, we find that, for all $t\geq 0$ sufficiently small, there is a unique  solution  $u_t\in C^{2,\alpha}(M)$ of
\begin{equation} \label{ut def}
\begin{dcases} \qquad
-\frac{4\,(n-1)}{(n-2)}\Delta_{g_t}\,u_t+R(g_t)\,u_t=0\qquad &\text{in }\operatorname{int}(M(\Sigma))\qquad\qquad\qquad\qquad,\\\qquad
\frac{2\,(n-1)}{n-2}\,D(g)_{\nu(\partial M,g_t)}u_t+H(\partial M,g_t)\,u_t=0&\text{on } M(\Sigma)\cap\partial M, \\\qquad
u_t=1& \text{on } \Sigma,
\end{dcases} 
\end{equation}
such that $(u_t-1)\in C_{\tau}^{2,\alpha}(M(\Sigma))$. Moreover, the limit
$$
\dot u=\lim_{t\searrow0}\,t^{-1}\,(u_t-1)
$$
exists in $ C_{\tau}^{2,\beta}(M(\Sigma)$ for every $\beta\in(0,\alpha)$; see \cite[pp.~73-74]{SchoenYau}.  By \eqref{ut def},
\begin{equation*}
	\begin{dcases} \qquad
		-\frac{4\,(n-1)}{(n-2)}\Delta_{g}\,\dot{u}+\lim_{t\searrow0} \,t^{-1}\,R(g_t)=0\qquad&\text{in }\operatorname{int}(M(\Sigma)),\qquad\qquad\qquad\qquad\\\qquad
		D(g)_{\nu(\partial M,g)}\dot{u}^i=0&\text{on } M(\Sigma)\cap\partial M, \\\qquad
		\dot{u}=0& \text{on } \Sigma.
	\end{dcases} 
\end{equation*}
  \indent  Let $$\hat g_t=(1+u_t)^\frac{4}{n-2}\,g_t.$$ Using that $(u_t-1)\in C_\tau^{2,\alpha}(M(\Sigma))$, we see that $\hat g_t$ is $C^2$-asymptotically flat and, using also Lemma \ref{conformal change},  that 
$R(\hat g_t)= 0$ in $M(\Sigma)$ and $H(\partial M,\hat g_t)= 0$ on $ M(\Sigma)\cap\partial M$.  
 Using Lemma \ref{perturbation}, we see that $\dot{u}$ is non-constant. By the maximum principle, $\dot{u}<0$ in $M(\Sigma)\setminus \Sigma$ and $D(g)_{\nu(\Sigma,g)}{\dot u}<0$ on $\Sigma$. Now, Lemma \ref{conformal change} and Lemma \ref{perturbation} imply that $H(\Sigma,\hat g_t)>0$ for all $t>0$ sufficiently small. As in the proof of Proposition \ref{ahf to conformally flat prop}, it follows that $(M,\hat g_t)$ has horizon boundary $\hat \Sigma_t\subset M(\Sigma)$ and that $\hat \Sigma_t\to \Sigma$ smoothly as $t\searrow 0$. \\ \indent On the one hand, using that $g_t=g$ on $\Sigma$, $\dot{u}=0$ on $\Sigma$, and $H(\Sigma,g)=0$, we conclude that 
\begin{align} \label{area constant} 
\lim_{t\searrow 0} t^{-1}\,\big(|\hat \Sigma_t|_{\hat g_t}-|\Sigma|_g\big)=0.
\end{align} 
Moreover, arguing as in the proof of Lemma \ref{scalar flat}, we have
$$
\lim_{t\searrow0}t^{-1}\,(m(\hat g_t)-m(g))=-\frac{2}{(n-2)\,\omega_{n-1}}\,\lim_{\lambda\to\infty}\lambda^{-1}\,\sum_{i=1}^n\int_{ \mathbb{R}^n_+\cap{S}^{n-1}_\lambda(0)}x^i\,\partial_i\dot u\,\mathrm{d}\mu(\bar g).
$$
In conjunction with Lemma \ref{asymptotic growth}, we conclude that
\begin{align} \label{mass decreases}  
\lim_{t\searrow0}t^{-1}\,(m(\hat g_t)-m(g))<0.
\end{align} 
On the other hand, by Corollary \ref{penrose coro}, we have
$$
\lim_{t\searrow0}\,t^{-1}\,\bigg(m(\hat g_t)-\left(\frac{1}{2}\right)^{\frac{n}{n-1}}\,\left(\frac{|\hat \Sigma_t|_{\hat g_t}}{\omega_{n-1}}\right)^{\frac{n-2}{n-1}}-m(g)+\left(\frac{1}{2}\right)^{\frac{n}{n-1}}\,\left(\frac{| \Sigma|_{ g}}{\omega_{n-1}}\right)^{\frac{n-2}{n-1}}\bigg)\geq 0
$$
This is not compatible with \eqref{area constant} and \eqref{mass decreases}.\\ \indent The assertion follows.
\end{proof} 
\begin{proof}[Proof of Theorem \ref{rigidity remark}]
	Suppose that $(M,g)$ is an asymptotically flat half-space with horizon boundary $\Sigma\subset M$  with $R(g)\geq 0$ in $M(\Sigma)$ and $H(\partial M,g)\geq 0$ on $ M(\Sigma)\cap\partial M$ such that
	$$
	m(g)=\left(\frac{1}{2}\right)^{\frac{n}{n-1}}\,\left(\frac{|\Sigma|_{g}}{\omega_{n-1}}\right)^{\frac{n-2}{n-1}}.
	$$
\indent 	Recall the definitions \eqref{tilde M} of the doubled manifold $(\tilde M,\tilde g)$ and \eqref{projection} of the projection $\pi:\tilde M\to M$. Moreover, recall that  $\tilde \Sigma=\pi^{-1}(\Sigma)$.
	By Lemma \ref{totally geodesic}, $h(\partial M,g)=0$ on $M(\Sigma)\cap\partial M$. In particular, $(\tilde M(\tilde \Sigma),\tilde g)$ is a $C^2$-asymptotically flat manifold with mass $\tilde m(\tilde g)=2\,m(g)$ and $\tilde \Sigma$ is a closed minimal surface  with $|\tilde \Sigma|_{\tilde g}=2\,|\Sigma|_{g}$. By symmetry, using that $M(\Sigma)$ is an exterior region, it follows that $\tilde M(\tilde \Sigma)$ is an exterior region.   By Theorem \ref{RPI rigidity}, $(\tilde M(\tilde \Sigma),\tilde g)$ is isometric to the exterior region of the Schwarzschild space \eqref{Schwarzschild space} with mass $\tilde m(\tilde g)$. It follows that $(M(\Sigma),g)$ is isometric to the exterior region of the Schwarzschild half-space \eqref{Schwarzschild half space} of mass $m(g)$. 
\end{proof}

\section{Rigidity in the Riemannian Penrose inequality}
In this section, we give an argument alternative to that in \cite{LuMiao} to show that the assumption that $(\tilde M,\tilde g)$ be spin in the rigidity statement of \cite[Theorem 1.4]{BrayLee}, stated here as Theorem \ref{RPI no boundary}, is not necessary. \\ \indent For the statement of Theorem \ref{RPI rigidity} below, recall from Appendix \ref{af appendix} the definition of an asymptotically flat manifold $(\tilde M,\tilde g)$, of its horizon boundary $\tilde \Sigma$, and of the exterior region $M(\tilde \Sigma)$. 
\begin{thm}
	\label{RPI rigidity} Let $(\tilde M,\tilde g)$ be an asymptotically flat manifold of dimension $3\leq n\leq7$ with horizon boundary $\tilde \Sigma\subset \tilde M$ such that $R(\tilde g)\geq 0$ in $M(\tilde \Sigma)$ and
	$$
	\tilde m(\tilde g)=\frac12\,\left(\frac{|\tilde \Sigma|_{\tilde g}}{\omega_{n-1}}\right)^{\frac{n-2}{n-1}}.
	$$
	Then $(\tilde M(\tilde \Sigma),\tilde g)$ is isometric to the exterior region of a Schwarzschild space \eqref{Schwarzschild space}.
\end{thm}
\begin{proof} 
Following the argument given in \cite[\S6]{BrayLee}, we aim to show that the manifold $(\hat M,\hat g)$ obtained by reflection of $(\tilde M,\tilde g)$ across $\tilde \Sigma$ is smooth so that the characterization of equality in the positive mass theorem, stated here as Theorem \ref{pmt no boundary}, applies to  $(\hat M,\hat u^{\frac{4}{n-2}}\,\hat g)$. Here, $\hat u\in C^2(\hat M)$ is the unique harmonic function that approaches $1$ respectively $0$ in the two ends of $(\hat M,\hat g)$. To this end, it suffices to show that $\tilde \Sigma$ is totally geodesic. \\ \indent 
	The argument presented in Lemma \ref{scalar flat} shows that $R(\tilde g)=0$. If $h(\tilde \Sigma,\tilde g)\neq 0$, the argument presented in Lemma \ref{perturbation} shows that there exists a family $\{\tilde g_t\}_{t\in[0,t_0)}$ of Riemannian metrics on $(\tilde M,\tilde g)$ such that
	\begin{itemize}
		\item[$\circ$] $\tilde g_t=\tilde g$ outside of a compact set,
		\item[$\circ$] $\tilde g_t\to \tilde g$ smoothly as $t\searrow 0$, 
		\item[$\circ$] $\lim_{t\searrow 0} t^{-1}(|\tilde \Sigma|_{\tilde g_t}-|\tilde \Sigma|_{\tilde g})=0$,  
		\item[$\circ$] $\lim_{t\searrow 0} t^{-1}\,H(\tilde \Sigma, \tilde g_t)= 0$,
	\end{itemize} 
	and 
	\begin{equation*}
		\begin{aligned}
			\lim_{t\searrow0}\,t^{-1}\,R(\tilde g_t)\geq 0 \text{ with strict inequality at some point}.
		\end{aligned}
	\end{equation*}
	Adapting the argument in the proof of Lemma \ref{totally geodesic} to the case of an asymptotically flat manifold, we see that this leads to a contradiction with the inequality in Theorem \ref{RPI no boundary}.
\end{proof} 
\begin{appendices} 
	\section{Asymptotically flat manifolds} \label{af appendix}
	In this section, we recall some facts about asymptotically flat manifolds. \\ \indent
			Let $ 3\leq n\leq 7$.  A metric $\tilde g$ on $\{\tilde x\in\mathbb{R}^n:|\tilde x|_{\bar g}>1/2\}$ is called $C^2$-asymptotically flat if its scalar curvature is integrable and if there is $\tau>(n-2)/2$ such that, as $\tilde x\to\infty$,
	\begin{equation} \label{asymptotically flat mf} 
	\begin{aligned}  
	&|\tilde g-\bar g|_{\bar g}+|x|_{\bar g}\,|D(\bar g)\tilde g|_{\bar g}+|x|^2_{\bar g}\,|D^2(\bar g)\tilde g|_{\bar g}=O\left(|\tilde x|_{\bar g}^{-\tau}\right).
	\end{aligned} 
	\end{equation} 
	\indent 
	A   complete connected Riemannian manifold $(\tilde M,\tilde g)$ of dimension $n$ is said to be an asymptotically flat manifold if the following properties all hold.
	\begin{itemize}
		\item[$\circ$]  $\tilde g$ is of class $C^2$.
		\item[$\circ$] There is a non-empty compact subset of $\tilde M$ whose complement  is diffeomorphic to the set $\{\tilde x\in\mathbb{R}^n:|\tilde x|_{\bar g}>1/2\}$.
		\item[$\circ$ ] The pull-back of $\tilde g$ by this diffeomorphism is $C^2$-asymptotically flat.
	\end{itemize}
\indent \indent	We usually fix such a diffeomorphism and refer to it as the asymptotically flat chart. 
	The mass of an asymptotically flat manifold is the quantity
	\begin{align} \label{mass}  
\tilde	m(\tilde g)=\lim_{\lambda\to\infty} \frac{1}{2\,(n-1)\,\omega_{n-1}}\,\lambda^{-1}\,\sum_{i,\,j=1}^n&\int_{S^{n-1}_\lambda(0)}\tilde x^i\,[(\partial_j\tilde{g})\big(e_i,e_j)-(\partial_i\tilde{g})(e_j,e_j)\big]\,\mathrm{d}\mu(\bar  g);
\end{align}
see \cite[p.~999]{ADM}. Here, $e_1,\dots,e_n$ are the canonical basis vectors of $\mathbb{R}^n$ and $\omega_{n-1}=|S^{n-1}_1(0)|_{\bar g}$ denotes the area of the $(n-1)$-dimensional unit sphere.
	R.~Bartnik has showed that the mass \eqref{mass} of a $C^2$-asymptotically flat manifold converges and does not depend on the choice of asymptotically flat chart; see \cite[Theorem 4.2]{Bartnik}. 
	\\ \indent
	Let $\tilde \Sigma\subset M$ be a  compact  hypersurface without boundary. We call the components of such a hypersurface closed. If $\tilde \Sigma$ is separating, we orient $\tilde \Sigma$ by the unit normal $\nu(\tilde \Sigma,\tilde g)$ pointing towards the closure $\tilde M(\tilde \Sigma)$ of the non-compact component of $\tilde M\setminus \tilde \Sigma$. The mean curvature 	$H(\tilde \Sigma,\tilde g)$ is then computed as the divergence of $-\nu(\tilde\Sigma,\tilde g)$ along $\tilde \Sigma$. \\ \indent      We say that $(\tilde M,\tilde g)$ has horizon boundary if there is a non-empty  hypersurface $\tilde \Sigma\subset \tilde M$ with the following two properties.
	\begin{itemize}
		\item[$\circ$]Each component of $\tilde \Sigma$ is a closed minimal hypersurface.
		\item[$\circ$] Every closed minimal hypersurfaces in $\tilde M(\tilde \Sigma)$ is a component of $\tilde \Sigma$.
	\end{itemize} 
If $(\tilde M,\tilde g)$ has horizon boundary $\tilde \Sigma\subset M$, we say that $\tilde M(\tilde \Sigma)$ is	 the exterior region of $\tilde M$ and we call the horizon $\tilde \Sigma$ an outermost minimal surface. An example of an exterior region with horizon boundary is the Schwarzschild space of mass $\tilde m>0$ and dimension $n\geq 3$ defined by 
\begin{align} \label{Schwarzschild space} 
(\tilde M(\tilde \Sigma),\tilde g)=\bigg(\bigg\{\tilde x\in\mathbb{R}^n:|\tilde x|_{\bar g}\geq\left(\frac{\tilde m}{2}\right)^{\frac{1}{n-2}}\bigg\},\bigg(1+\frac{\tilde m}{2}\,|\tilde x|_{\bar g}^{2-n}\bigg)^\frac{4}{n-2}\,\bar g\bigg).
\end{align}
where
$$
\tilde \Sigma=\left\{\tilde x\in\mathbb{R}^n:|\tilde x|_{\bar g}=\left(\frac{\tilde m}{2}\right)^{\frac{1}{n-2}}\right\}.
$$
 \indent
The positive mass theorem has been  proved by R.~Schoen and S.-T.~Yau in \cite{SchoenYau}  using minimal surface techniques and subsequently by E.~Witten in \cite{Witten} using certain solutions of the Dirac equation.  
	\begin{thm}[{\cite[Theorem 4.2]{schoenvariational}}] \label{pmt no boundary} Let $(\tilde M,\tilde g)$ be an asymptotically flat manifold of dimension $ 3\leq n\leq7$ whose scalar curvature is non-negative. There holds $\tilde m(\tilde g)\geq 0$. Moreover, $\tilde m(\tilde g)=0$  if and only if $(\tilde M,\tilde g)$ is isometric to $(\mathbb{R}^n,\bar g)$.
	\end{thm}
The Riemannian Penrose inequality has been proved by G.~Huisken and T.~Ilmanen in the case where the horizon boundary is connected using inverse mean curvature flow in \cite{HI}. For general horizon boundary, it has been obtained by H.~L.~Bray \cite{Bray} using his quasi-static flow. H.~L.~Bray's technique has been extended to higher dimensions  in his joint work \cite{BrayLee} with D.~A.~Lee.
\begin{thm}[{\cite[Theorem 1.4]{BrayLee}}] \label{RPI no boundary} Let $(\tilde M,\tilde g)$ be an asymptotically flat manifold of dimension $3\leq n\leq7$ with horizon boundary $\tilde \Sigma\subset \tilde M$ such that $R(\tilde g)\geq 0$ in $M(\tilde \Sigma)$. There holds
	\begin{align} \label{RPI}  
	\tilde m(\tilde g)\geq\frac12\,\left(\frac{|\tilde \Sigma|_{\tilde g}}{\omega_{n-1}}\right)^{\frac{n-2}{n-1}}.
	\end{align} 
	If $(\tilde M,\tilde g)$ is a  spin manifold, equality holds if and only if $(\tilde M(\tilde \Sigma),\tilde g)$ is isometric to the exterior region of a Schwarzschild space \eqref{Schwarzschild space}.
\end{thm}
\begin{rema}
	The assumption that $(\tilde M,\tilde g)$ be spin is not necessary; see \cite[Theorem 1.1]{LuMiao} and Theorem \ref{RPI rigidity}.
\end{rema}
\begin{rema}
	 G.~Lam \cite[Corollary 20]{Lam} and L.-H.~Huang and D.~Wu \cite[Theorem 2]{HuangWu} have showed that Theorem \ref{RPI no boundary} holds in all dimensions if $(\tilde M,\tilde g)$ is an asymptotically flat hypersurface of a Euclidean space.
\end{rema}
	\section{Riemannian geometry}
	In this section, we recall some facts from Riemannian geometry.  
\begin{lem}[{\cite[\S3]{kazdanwarner}}] \label{conformal change}
Let $(M,g)$ be a Riemannian manifold of dimension $n\geq 3$ and  $u\in C^\infty(M)$ be a positive function. Let
$$
g^u=u^{\frac{4}{n-2}}\,g
$$	
and suppose that $\Sigma\subset M$ is a two-sided  hypersurface with unit normal $\nu(\Sigma,g)$ and mean curvature $H(\Sigma,g)$ computed as the divergence of $\nu(\Sigma,g)$ along $\Sigma$.	\\ 
	\indent There holds
$$
R(g^u)=u^{-\frac{n+2}{n-2}}\,\left(-\frac{4\,(n-1)}{n-2}\Delta_g u+R(g)\,u\right)
$$
and
$$
H(\Sigma,g^u)=u^{-\frac{n}{n-2}}\left(\frac{2\,(n-1)}{n-2}\,D(g)_{\nu(\Sigma,g)}u+H(\Sigma,g)\,u\right).
$$
	\end{lem}
	\begin{lem} \label{geodesic lemma}
	Let $n\geq 2$.	There exists  $\varepsilon>0$ with the following property. Suppose that $\sigma$ is a symmetric $(0,2)$-tensor on $\mathbb{R}^n_+\cap B^n_1(0)$ with
		\begin{align} \label{geodesic lemma sigma} 
		|\sigma|_{\bar g}+|D(\bar g)\sigma|_{\bar g}+|D^2(\bar g)\sigma|_{\bar g}<\varepsilon
		\end{align} 
		and let $g=\bar g+\sigma$. The map
		$$
		\{y\in\mathbb{R}^{n-1}:|y|_{\bar g}<1/2\}\times[0,1/2)\to \mathbb{R}^n_+\cap B^n_1(0)\qquad\text{given by}\qquad (y,t)\mapsto\exp(g)_{y}({t\,\nu(\partial \mathbb{R}^n_+,g)(y)})
		$$
		is injective. 
	\end{lem}
\begin{proof}
	Let $y_1,\,y_2\in\mathbb{R}^{n-1}$ with $|y_1|_{\bar g},\,|y_2|_{\bar g}<1/2$. Let 
	$$
	\omega_1(t)=\exp(g)_{y_1}({t\,\nu(\partial \mathbb{R}^n_+,g)(y_1)})
	\qquad\text{and} \qquad 
\omega_2(t)=\exp(g)_{y_2}({t\,\nu(\partial \mathbb{R}^n_+,g)(y_2)}).
$$
Moreover, let
	$s\in [0,1/2]$ be maximal such that, for all $t\in[0,s]$, 
	\begin{equation*} 
	\begin{aligned} 
&\circ\qquad \omega_1(t),\,\omega_2(t)\in \mathbb{R}^n_+\cap B^n_{3/4}(0),\\
&\circ\qquad	\frac12\,|y_1-y_2|_{\bar g}\leq |\omega_1(t)-\omega_2(t)|_{\bar g}\leq 2\,|y_1-y_2|_{\bar g},\text{ and}\\
&\circ \qquad |\dot \omega_1(t)-\dot \omega_2(t)|_{\bar g}\leq \frac12\,|y_1-y_2|_{\bar g}.
	\end{aligned} 
	\end{equation*}
 By  \eqref{geodesic lemma sigma}, as $\varepsilon\searrow 0$,
	$$
	|\dot \omega_1(0)-\dot \omega_2(0)|_{\bar g}=O(\varepsilon)\,|y_1-y_2|_{\bar g}.
	$$
In particular, $s>0$. By the geodesic equation,
	$$
	|\ddot \omega_1(t)-\ddot \omega_2(t)|_{\bar g}=O(\varepsilon)\,|y_1-y_2|_{\bar g}
	$$
	on $[0,s]$. Integrating, it follows that $s= 1/2$ provided that $\varepsilon>0$ is sufficiently small. \\ \indent The assertion follows.
\end{proof}
\section{Laplace operator on asymptotically flat half-spaces with horizon boundary} 
In \cite[Proposition 3.3]{ABDL}, S.~Almaraz, E.~Barbosa, and L.~L.~de Lima have proved an existence and uniqueness result for the Laplace equation on  asymptotically flat half-spaces. In this section, we explain how their result can be adapted to an asymptotically flat half-space with horizon boundary. \\ \indent 
Let $(M,g)$ be an asymptotically flat half-space of rate $\tau>(n-2)/2$ with horizon boundary. \\ \indent  We fix an asymptotically flat chart $\Phi:\{x\in\mathbb{R}^n_+:|x|_{\bar g}>1/2\}\to M$. We may assume that $\operatorname{Im}(\Phi)\cap \Sigma=\emptyset$. Let $K\subset M$ be the connected compact set with $\partial K=\Sigma\cup \Phi(\{x\in\mathbb{R}^n_+:|x|_{\bar g}\geq 2\})$. \\ \indent Given $\alpha\in(0,1)$, an integer $k\geq 0$, and $u\in C_{loc}^{k,\alpha}(M)$, we adapt from \cite[\S3]{ABDL} the definition of the weighted Hölder norm 
\begin{equation} \label{weighted hoelder}
\begin{aligned}
&|u|_{C^{k,\alpha}_\tau(M(\Sigma))}
\\&\quad =\,|u|_{C^{k,\alpha}(K)}+\sum_{i=0}^k\sup_{|x|_{\bar g}>1}|x|_{\bar g}^{i+\tau}\,|(D^i(\bar g)u)(x)|_{\bar g}\\&\quad \qquad  +\sup_{|x_1|_{\bar g}>1}\,\,\sup_{2\,|x_2-x_1|_{\bar g}<|x_1|_{\bar g}}|x_1|_{\bar g}^{k+\tau+\alpha}\,|x_1-x_2|_{\bar g}^{-\alpha}\, |(D^k(\bar g)u)(x_1)-(D^k(\bar g)u)(x_2)|_{\bar g}
\end{aligned}
\end{equation}
where $x,\,x_1,\,x_2\in \mathbb{R}^n_+$. We  define $$C^{k,\alpha}_\tau(M(\Sigma))=\{u\in C_{loc}^{k,\alpha}(M):|u|_{C^{k,\alpha}_\tau(M(\Sigma))}<\infty\}.$$
 \indent    Likewise, given $\beta\in(0,1)$, an integer $\ell\geq 1$, and $a\in C_{loc}^{\ell,\beta}(\partial M)$, we define 
\begin{equation*}
\begin{aligned}
|a|_{C^{\ell,\beta}_\tau( M(\Sigma)\cap\partial M)} =&\,|a|_{C^{\ell,\beta}( K\cap\partial M)}+\sum_{j=0}^\ell\sup_{|y|_{\bar g}>1}|y|_{\bar g}^{j+\tau}\,|(D^j(\bar g)a)(y)|_{\bar g}\\&\quad +\sup_{|y_1|_{\bar g}>1}\,\,\sup_{2\,|y_1-y_2|_{\bar g}<|y_1|_{\bar g}}|y_1|_{\bar g}^{\ell+\tau+\beta}\,|y_1-y_2|_{\bar g}^{-\beta}\, |(D(\bar g)^\ell a)(y_1)-(D(\bar g)^\ell a)(y_2)|_{\bar g}
\end{aligned}
\end{equation*} 
where $y,\,y_1,\,y_2\in\mathbb{R}^{n-1}\times\{0\}$. 
 We define 
 $$C_\tau^{\ell,\beta}( M(\Sigma)\cap\partial M)=\big\{a\in C_{loc}^{\ell,\beta}(\partial M):|a|_{C^{\ell,\beta}_\tau( M(\Sigma)\cap\partial M)}<\infty\text{ and } D(g|_{\partial M})_{\nu(\Sigma,g)}a=0\text{ on } \Sigma\cap\partial M\big\}.$$
\begin{prop} \label{PDE prop}
	Let $\alpha\in(0,1)$. There exists a constant $c>0$ with the following property. Given $\psi\in C^{0,\alpha}_{\tau}(M(\Sigma))$ and $a\in C^{1,\alpha}_{\tau}( M(\Sigma)\cap\partial M)$, there exists a unique solution $v\in C_\tau^{2,\alpha}(M(\Sigma))$ of
	\begin{equation} \label{neumann problem} 
	\begin{dcases}\qquad
	-\Delta_gv-\psi=0\qquad&\text{in }\operatorname{int}(M(\Sigma)),\qquad\qquad\qquad\qquad\qquad\qquad\qquad\qquad\\\qquad
	D(g)_{\nu(\partial M,g)}v-a=0&\text{on }  M(\Sigma)\cap\partial M,\text{ and}\\\qquad
	D(g)_{\nu(\Sigma,g)}v=0&\text{on } \Sigma.
	\end{dcases}
	\end{equation}
	There holds
	\begin{align} \label{hoelder estimate} 
	|v|_{C_{\tau}^{2,\alpha}(M(\Sigma))}\leq c\,\big(|\psi|_{C^{0,\alpha}_{\tau}(M(\Sigma))}+|a|_{C^{1,\alpha}_{\tau}( M(\Sigma)\cap\partial M)}\big).
	\end{align} 
\end{prop}
\begin{proof} 
	The case where $\Sigma=\emptyset$ has been proved in \cite[Proposition 3.3]{ABDL}. \\ \indent 
	If $\Sigma\subset M$ is a compact hypersurface whose components are closed hypersurfaces or free boundary hypersurfaces, we consider the differentiable manifold $\hat M=M(\Sigma)\times \{-1,\,1\}/\sim$ where
	\begin{itemize}
		\item[$\circ$] $(x_1,\pm1)\sim (x_2,\pm1)$ if and only if $x_1=x_2$ and
	\item[$\circ$]  $(x_1,\pm1)\sim (x_2,\mp1)$ if and only if $x_1,\,x_2\in\Sigma$ and $x_1=x_2.$
\end{itemize}
	  We define the Riemannian metric $\hat g$ on $\hat M$ by $\hat g(\hat x)=g(\pi(\hat x))$ where $\pi([(x,\pm1)])=x$.  Note that $(\hat M,\hat g)$ has two asymptotically flat ends and that $\hat g$ is of class $C^2$ away from $\pi^{-1}(\Sigma)$. Moreover, note that, although $\hat g$ is only Lipschitz, the coefficients of $\Delta_{\hat g}$ are still Lipschitz since $\Sigma$ is minimal. 
\\ \indent  Let $\hat \psi:\hat M\to\mathbb{R}$ be given by $\hat \psi(\hat x)=\psi(\pi(\hat x))$ and $\hat a: \partial\hat  M\to\mathbb{R}$ be given by $\hat a(\hat x)=a(\pi(\hat x))$. Clearly, $\hat \psi\in C^{0,\alpha}_\tau(\hat M)$. Moreover, using that $D(g|_{\partial M})_{\nu(\Sigma,g)}a=0$ on $ \Sigma\cap\partial M$, we conclude that $\hat a \in C^{1,\alpha}_{\tau}(\partial \hat M)$. By \cite[Proposition 3.3]{ABDL}, there exists a unique $\hat v\in C_\tau^{2,\alpha}(\hat M)$ with 
	\begin{equation} \label{Neumann}  
\begin{dcases}\qquad
-\Delta_{\hat g}\hat v-\hat \psi=0\qquad&\text{in }\hat M,\qquad\qquad\qquad\qquad\qquad\qquad\qquad\qquad\qquad\\\qquad
D(\hat g)_{\nu(\partial \hat M,\hat g)}\hat v-\hat a=0&\text{on } \partial \hat M
\end{dcases}
\end{equation}
and
	\begin{align*}  
|\hat v|_{C_{\tau}^{2,\alpha}(\hat M)}\leq c\,\big(|\hat \psi|_{C^{0,\alpha}_{\tau}(\hat M)}+|\hat a|_{C^{1,\alpha}_{\tau}(\partial \hat M)}\big).
\end{align*} 
Here, $c>0$ is a constant independent of $\hat \psi$ and $\hat a$. Moreover, by uniqueness, $\hat v([x,1])=\hat v([x,-1])$ for all $x\in M$. It follows that $v:M(\Sigma)\to\mathbb{R}$ given by $v(x)=\hat v([x,1])$ satisfies \eqref{neumann problem} and \eqref{hoelder estimate}.\\ \indent The assertion follows.
\end{proof}

\begin{prop} \label{PDE prop 2}
	Let $\alpha\in(0,1)$ and $\chi\in C^{0,\alpha}_{\tau}(M(\Sigma))$ with $\chi\geq 0$. There exists a constant $c>0$ with the following property. Given $\psi\in C^{0,\alpha}_{\tau}(M(\Sigma))$ and $b\in C^{2,\alpha}( \Sigma )$ with $D(g)_{\nu(\partial M,g)}b=0$ on $\partial \Sigma$, there exists a unique solution $v\in C_\tau^{2,\alpha}(M(\Sigma))$ of
	\begin{equation*} 
	\begin{dcases}\qquad
	-\Delta_gv+\chi\,v-\psi=0\qquad&\text{in }\operatorname{int}(M(\Sigma)),\qquad\qquad\qquad\qquad\qquad\qquad\qquad\qquad\\\qquad
	D(g)_{\nu(\partial M,g)}v=0&\text{on }  M(\Sigma)\cap\partial M,\text{ and}\\\qquad
	v-b=0&\text{on } \Sigma.
	\end{dcases}
	\end{equation*}
	There holds
	\begin{align*}
	|v|_{C_{\tau}^{2,\alpha}(M(\Sigma))}\leq c\,\big(|\psi|_{C^{0,\alpha}_{\tau}(M(\Sigma))}+|b|_{C^{2,\alpha}( \Sigma)}\big).
	\end{align*} 
\end{prop}
\begin{proof}
The proof is very similar to that of Proposition \ref{PDE prop} and we only sketch the necessary modifications. First, we consider an appropriate Dirichlet problem on the double \eqref{tilde M} of $(M,g)$ instead of the Neumann problem \eqref{Neumann} on  $(\hat M,\hat g)$. Second,  we use \cite[Proposition 2.2]{Bartnik} and \cite[Theorem 9.2]{LeeParker} instead of \cite[Proposition 3.3]{ABDL}.
\end{proof}
\section{Local perturbations of a Riemannian metric}
In this section, we construct  local perturbations of a Riemannian metric that are used in this paper. 
\begin{lem}\label{f conformal}
	Let $(M,g)$ be a Riemannian manifold of dimension $n\geq 3$ with boundary $\partial M$ oriented by the unit normal $\nu(\partial M,g)$ pointing towards $M$. Let  $\Sigma\subset M$ be a compact  hypersurface whose components are either closed  or free boundary hypersurfaces. 
	There exists a sequence $\{\psi_i\}_{i=1}^\infty$ of functions $\psi_i\in C^\infty(M)$ with the following properties:
	\begin{itemize}
		\item[$\circ$] $D(g)_{\nu(\Sigma,g)}\psi_i<0$ on $\Sigma$.
		\item[$\circ$] $D(g)_{\nu(\partial M,g)}\psi_i\leq 0$ on  $\partial M$.
		\item[$\circ$] $|\psi_i|+|D(g)\psi_i|_g+|D^2(g)\psi_i|_g=o(1)$ in $M$, as $i\to\infty$.
	\end{itemize}
	Moreover, if $W\subset M$ is a  neighborhood of $\Sigma$, then $\operatorname{spt}(\psi_i)\subset W$ for all but finitely many $i$. 
\end{lem}
\begin{proof}
	Let $\alpha\in C^\infty(\mathbb{R})$ be such that
	\begin{itemize}
		\item[$\circ$]  $\alpha(0)=0$,
		\item[$\circ$] $\alpha'(0)=-1$, and
		\item[$\circ$] $\alpha(t)=0$ if $|t|\geq 1$.
	\end{itemize}
	Let $\beta\in C^\infty(\mathbb{R})$ be non-negative such that 
	\begin{itemize}
		\item[$\circ$]  $\beta(t)=1$ if $|t|\leq 1$ and
		\item[$\circ$] $\beta(t)=0$ if $|t|\geq 2$.
		
	\end{itemize}
	Let $a,\,b:M\to\mathbb{R}$ be given by $a(x)=\operatorname{dist}(x,\Sigma,g)$ and $b(x)=\operatorname{dist}(x,\partial M,g)$, respectively. Here, $\operatorname{dist}(\,\cdot\,,\Sigma,g)$ and $\operatorname{dist}(\,\cdot\,,\partial M,g)$ are the signed distance functions that become positive in direction of the respective unit normals.  Given an integer $i\geq 1$, we define $\psi_i:M\to\mathbb{R}$ by
	$$
	\psi_i(x)=i^{-3}\,\big[\alpha(i\,a(x))+\beta(i\,a(x))\,\alpha(i\,b(x))\big].
	$$ 
	Note that $\psi_i$ is smooth provided that $i$ is sufficiently large and  that, as $i\to\infty$, 
	$$
	|\psi_i|+|D(g)\psi_i|_g+|D^2(g)\psi_i|_g=o(1).
	$$
	Moreover, if $W\subset M$ is a neighborhood of $\Sigma$, then $\operatorname{spt}(\psi_i)\subset W$ for all but finitely many $i$.	
	\\ \indent On $\Sigma$, there holds
	$$
i^2\,	D(g)_{\nu(\Sigma,g)}\psi_i= -1+\alpha'(i\,b(x))\,(D(g)_{\nu(\Sigma,g)}b)(x).
	$$
	Using that $\nu(\Sigma,g)(y)\in T_y\partial M$ for every $y\in\partial \Sigma$, we obtain $D(g)_{\nu(\Sigma,g)}b=0$ on $\partial \Sigma$. Consequently, $(D(g)_{\nu(\Sigma,g)}b)(x)=O(b(x))$ on $\Sigma$. Using that $\alpha'(i\,b(x))=0$ if $i\,b(x)\geq 1$, we conclude that
	$$
i^{2}\,	D(g)_{\nu(\Sigma,g)}\psi_i\leq-\frac12,
	$$
	provided that $i$ is sufficiently large. 
	\\ \indent 	
	On $\partial M$, we have
	$$
	i^{2}\,D(g)_{\nu(\partial M,g)}\psi_i=[\alpha'(i\,a(x))\,(D(g)_{\nu(\partial M,g)}a)(x)-\beta(i\,a(x))]
	$$
	If $i\,a(x)\geq 1$, we have
	$$
i^{2}\,	D(g)_{\nu(\partial M,g)}\psi_i=-\,\beta(i\,a(x))\leq 0.
	$$
	If $i\,a(x)<1$, we have, as before, $(D(g)_{\nu(\partial M,g)}a)(x)=O(a(x))$ while $\beta(i\,a(x))=1$. Consequently, 
	$$
i^{2}\,	D(g)_{\nu(\partial M,g)}\psi_i\leq-\frac12
	$$
	provided that $i$ is sufficiently large. \\
	\indent The assertion follows.
\end{proof} 
\begin{lem} Let $(M,g)$ be a Riemannian manifold of dimension $n\geq 3$ with boundary $\partial M$. Let $U\subset M$  open be such that $U\cap \partial M\neq \emptyset$.  There exists a function $\psi\in C^\infty(M)$ with the \label{scalar curvature for mean curvature} following properties. \begin{itemize}
	\item[$\circ$]	$\psi\geq 0$ in $M$ and $\operatorname{spt}(\psi)\subset U$. 
	\item[$\circ$] $\Delta_g\psi \geq 0$ in $M$ and $\Delta_g\psi>0$ at some point.
	\end{itemize}
\end{lem}
\begin{proof}
	Note that there is a constant $c>0$ such that the following holds. For every $\varepsilon>0$ sufficiently small, there exists a  map $\Phi_\varepsilon:B^n_{2\,\varepsilon}(0)\cap\mathbb{R}^n_+\to U$ such that
	\begin{figure}\centering
		\includegraphics[width=0.7\linewidth]{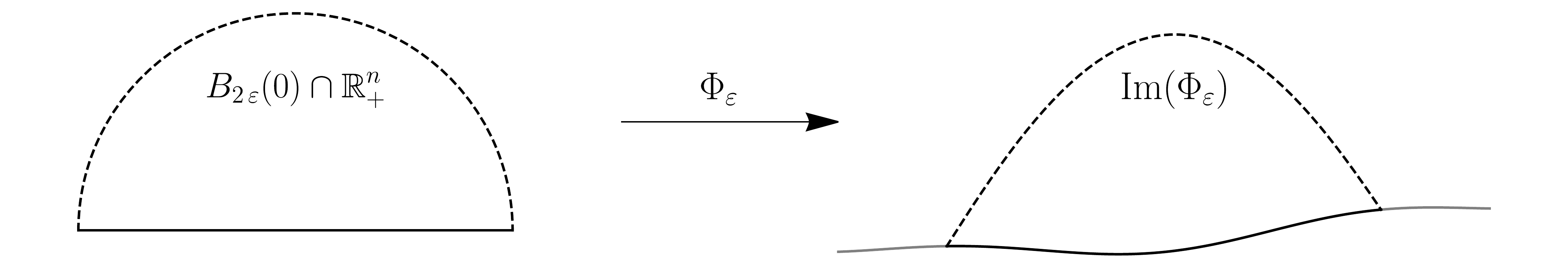}
		\caption{An illustration of the map $\Phi_\varepsilon$. $B^n_{2\,\varepsilon}(0)\cap \partial\mathbb{R}^n_+$ is presented by the solid black line on the left. The map $\Phi_\varepsilon$ maps $B^n_{2\,\varepsilon}(0)\cap \partial\mathbb{R}^n_+$ to a subset of $U\cap \partial M$ which is presented by the solid black line on the right.   $\partial M$ is showed by the solid gray line.  }
		\label{parametrization}
	\end{figure}
	\begin{itemize}
		\item[$\circ$] $\Phi_\varepsilon$ is an embedding,
		\item[$\circ$] $\Phi_\varepsilon^{-1}(U\cap \partial M)=\partial \mathbb{R}^n_+\cap B^n_{2\,\varepsilon}(0)$, 
		\item[$\circ$] $|\Phi_{\varepsilon}^*g-\bar g|_{\bar g}\leq c\,\varepsilon$, and
		\item[$\circ$] $|D(\bar g)\Phi_\varepsilon^*g|_{\bar g}\leq c$;
	\end{itemize}
	see Figure \ref{parametrization}. \\
	\indent 	Let $f_\varepsilon:\mathbb{R}_+^n\to \mathbb{R}$ be given by $$f_\varepsilon(x)=\begin{dcases}
	e^{-(n+2)\,(2\,\varepsilon-|x+\varepsilon\,e_n|_{\bar g})^{-1}}\quad &\text{if }|x+\varepsilon\,e_n|_{\bar g}< 2\,\varepsilon,\\
	0&\text{else}.
	\end{dcases}
	$$
	When $|x+\varepsilon\,e_n|_{\bar g}<2\,\varepsilon$, we compute
	\begin{align*} 
	\Delta_{\bar g} f_\varepsilon=&(n+2)^2\,(2\,\varepsilon-|x+\varepsilon\,e_n|_{\bar g})^{-4}\,f_\varepsilon\\&\qquad -2\,(n+2)\,(2\,\varepsilon-|x+\varepsilon\,e_n|_{\bar g})^{-3}\,f_{\varepsilon}\\&\qquad -(n-1)\,(n+2)\,(2\,\varepsilon-|x+\varepsilon\,e_n|_{\bar g})^{-2}\,|x+\varepsilon\,e_n|^{-1}_{\bar g}\,f_{\varepsilon}.
	\end{align*} 
	Since $|x+\varepsilon\,e_n|_{\bar g}\geq \varepsilon$ on $\mathbb{R}^n_+$, it follows that
	$$
	\Delta_{\bar g} f_\varepsilon\geq (n+2)\,(2\,\varepsilon-|x+\varepsilon\,e_n|_{\bar g})^{-4}\,f_{\varepsilon}.
	$$
	Likewise, 
	$$
	|D({\bar g}) f_\varepsilon|_{\bar g}\leq (n+2)\,\varepsilon^2\,(2\,\varepsilon-|x+\varepsilon\,e_n|_{\bar g})^{-4}\,f_{\varepsilon}$$ and $$
	|D^2(\bar g) f_\varepsilon|_{\bar g}\leq \sqrt{n}\,(n+2)^2\,(2\,\varepsilon-|x+\varepsilon\,e_n|_{\bar g})^{-4}\,f_{\varepsilon}.
	$$
	Let $\psi_\varepsilon: M\to\mathbb{R}$ be given by $$\qquad \psi_\varepsilon(x)=\begin{dcases}&f_\varepsilon(\Phi_\varepsilon^{-1}(x)) \qquad\text{if }x\in\operatorname{Im}(\Phi_\varepsilon),
	\\&0\qquad\qquad\qquad\,\,\text{else}.
	\end{dcases}$$
	Note that $\psi_\varepsilon\in C^\infty(M)$ and $\operatorname{spt}(\psi_\varepsilon)\subset U$. Moreover, 
	$
	\Delta_g \psi_\varepsilon\geq 0$ and $\Delta_g \psi_\varepsilon> 0$ at some point	provided that $\varepsilon>0$ is sufficiently small. 
	\\ \indent
	The assertion follows.
\end{proof}
\begin{lem} \label{scalar curvature balance}
	There exist a function $\rho\in C^\infty(\mathbb{R}^{n-1})$ and a constant $c>0$ with the following properties. 
	\begin{itemize}
		\item[$\circ$] $\rho(x)=0$ for all $x\in\mathbb{R}^{n-1}$ with $|x|_{\bar g}\geq 1$.
\item[$\circ$] $\rho(x)>0$ for all $x\in\mathbb{R}^{n-1}$ with $|x|_{\bar g}< 1$.		
\item[$\circ$]  $|\rho(x)|+|(D(\bar g)\rho)(x)|_{\bar g}+|(D^2( \bar g)\rho)(x)|_{\bar g}\leq c\, (\Delta_{\bar g} \rho)(x)$ for all $x\in\mathbb{R}^{n-1}$ with $1/2<|x|_{\bar g}< 1$.
	\end{itemize}
\end{lem}
\begin{proof}
Let $\eta\in C^\infty(\mathbb{R})$ be a function with 
	\begin{itemize}
		\item[$\circ$] $\eta(s)=0$ if $s\geq1$,
		\item[$\circ$] $\eta(s)=e^{-(n+1)\,(1-s)^{-1}}$ if $s\in[1/2,1)$,
			\item[$\circ$] $\eta(s)>0$ if $s\in(1/4,1/2)$, and
			\item[$\circ$] $\eta(s)=1$ if $s\leq1/4$.
	\end{itemize}
By a direct computation as in the proof of Lemma \ref{scalar curvature for mean curvature}, the function $\rho:\mathbb{R}^{n-1}\to\mathbb{R}$ given by $\rho(x)=\eta(|x|_{\bar g})$ satisfies the asserted properties.
\end{proof} 

\section{Asymptotic growth estimate for subharmonic functions}
In this section, we derive an asymptotic growth estimate for subharmonic functions on asymptotically flat half-spaces. The corresponding estimate for subharmonic functions on asymptotically flat manifolds has been stated by J.~Corvino in \cite[p.~164]{Corvino} and proved in detail by S.~Czimek in \cite[Proposition 2.6]{czimek2014static}. We note that the argument in \cite[Proposition 2.6]{czimek2014static} can be adapted to the setting of an asymptotically flat half-space. Below, we give a different, self-contained proof. \\ \indent Note that, in some sense, Lemma \ref{asymptotic growth} is a quantitative version of the Hopf boundary point lemma as stated in, e.g.\,, \cite[Lemma 3.4]{GT}.
\begin{lem} \label{asymptotic growth}
	Let $n\geq 3$ and $g$ be a $C^2$-asymptotically flat metric on $\mathbb{R}^n_+$.
	Suppose that there are a negative function $u\in C^{2,\alpha}_{\tau}(\mathbb{R}^n_+)$ and a number $\lambda_0>1$ such that
	\begin{itemize}
		\item[$\circ$] 	 $\Delta_gu\geq 0$ in $\mathbb{R}^n_+\setminus B^n_{\lambda_0}(0),$
\item[$\circ$]  $\Delta_gu$ is integrable, and
\item[$\circ$] $D(g)_{\nu(\partial\mathbb{R}^n_+,g)}u=0$ on $ \partial\mathbb{R}^n_+\setminus B^n_{\lambda_0}(0)$.
	\end{itemize}
  Then 
	$$
	\lim_{\lambda\to\infty}\lambda^{-1}\,\sum_{i=1}^n\int_{\mathbb{R}^n_+\cap S^{n-1}_\lambda(0)} x^i\,\partial_iu\,\mathrm{d}\mu(\bar g)>0.
	$$
\end{lem}
\begin{proof} We first assume that $\Delta_gu=0$. \\ \indent 
Let $\eta\in C^\infty(\mathbb{R})$ be a non-negative function with $\eta'(0)=1$ and $\eta(s)=0$ if $s\geq 1/2$. Let $f:\mathbb{R}^n_+\setminus\{0\}\to\mathbb{R}$ be given by 
	$$
	f(x)=-(\log|x|_{\bar g})^{-1}\,|x|_{\bar g}^{-(n-2)}+\eta(|x|_{\bar g}^{-1}\,x_n)\,|x|_{\bar g}^{-(n-2)-\tau}.
	$$
	We compute
	$$
	\Delta_{\bar g}f=-(2+(n-2)\,\log|x|_{\bar g})\,(\log|x|_{\bar g})^{-3}\,|x|_{\bar g}^{-n}+O(|x|_{\bar g}^{-n-\tau}).
	$$
	Likewise,
	$$
D(\bar g)f=O((\log|x|_{\bar g})^{-1}\,|x|_{\bar g}^{-(n-1)})\qquad\text{and}\qquad 
	D^2(\bar g)f=O((\log|x|_{\bar g})^{-1}\,|x|_{\bar g}^{-n}).
	$$
	Moreover, on $\partial \mathbb{R}^n_+$, we have
	$$
	D(\bar g)_{e_n}f=-|x|_{\bar g}^{-(n-1)-\tau}.
	$$
	 Increasing $\lambda_0>1$  if necessary, we find that   $\Delta_gf<0$ in $\mathbb{R}^n_+\setminus B^n_{\lambda_0}(0)$ and $D(g)_{\nu(\partial\mathbb{R}^n_+,g)}f<0$ on $\partial\mathbb{R}^n_+\setminus B^n_{\lambda_0}(0)$. Let $\delta>0$ be such that $u<\delta\,f$ on $\mathbb{R}^n_+\cap S^n_{\lambda_0}(0)$. By the maximum principle, for every $\lambda>\lambda_0$,
	 $$
	u< \delta\,(f-\inf\{f(x):x\in\mathbb{R}^n_+\cap S^n_\lambda(0)\})
	 $$
	 in $\mathbb{R}^n_+\cap({B}^n_\lambda(0)\setminus B^n_{\lambda_0}(0))$. Letting $\lambda\to\infty$, we conclude that  	 
	 \begin{equation} \label{max princ est} 
	 u<\delta\, f\qquad\text{in }\mathbb{R}^n_+\setminus B^n_{\lambda_0}(0).
	 \end{equation}  \indent 
	Note that
	\begin{equation*}
	\lambda^{-1}\,\sum_{i=1}^n\,\int_{\mathbb{R}^n_+\cap S^{n-1}_\lambda(0)} x^i\,\partial_iu\,\mathrm{d}\mu(\bar g)=\int_{\mathbb{R}^n_+\cap S^{n-1}_\lambda(0)} D(g)_{\nu(\mathbb{R}^n_+\cap S^{n-1}_\lambda(0),g)}u\,\mathrm{d}\mu(g)+O(\lambda^{(n-2)-2\,\tau}).
	\end{equation*}
	By the divergence theorem,  the limit 
	\begin{equation} \label{integral exists}
	z=\lim_{\lambda\to\infty}\lambda^{-1}\,\sum_{i=1}^n\,\int_{\mathbb{R}^n_+\cap S^{n-1}_\lambda(0)} x^i\,\partial_iu\,\mathrm{d}\mu(\bar g)
	\end{equation}
	exists and 
		\begin{equation} \label{z est}
	\lambda^{-1}\,\sum_{i=1}^n\,\int_{\mathbb{R}^n_+\cap S^{n-1}_\lambda(0)} x^i\,\partial_iu\,\mathrm{d}\mu(\bar g)=z+O(\lambda^{(n-2)-2\,\tau}).
	\end{equation}  \indent 
	Suppose, for a contradiction, that $z\leq 0$. 	Let $w:[\lambda_0,\infty)\to\mathbb{R}$ be given by 
	$$
	w(\lambda)=\lambda^{-1}\,\int_{\mathbb{R}^n_+\cap S^{n-1}_\lambda(0)}u\,\mathrm{d}\mu(\bar g)
	$$
	and note that, using also \eqref{z est},
	$$
	w'=(n-2)\,\lambda^{-1}\,w+\lambda^{-2}\,\sum_{i=1}^n\,\int_{\mathbb{R}^n_+\cap S^{n-1}_\lambda(0)} x^i\,\partial_iu\,\mathrm{d}\mu(\bar g)\leq(n-2)\, \lambda^{-1}\,w+O(\lambda^{(n-3)-2\,\tau}).
	$$
On the one hand, using \eqref{max princ est}, we have 
	$$
	w\leq -\frac12\,\delta\,\omega_{n-1}\,(\log(\lambda))^{-1}
	$$
	and hence $w'\leq(n-2)/2\,\lambda^{-1}\,w$
	for every $\lambda\geq\lambda_0$ provided that $\lambda_0>1$ is sufficiently large. It follows that
	\begin{equation} \label{w upper estimate}
	w(\lambda)\leq \lambda_0^{-(n-2)/2}\,\lambda^{(n-2)/2}\,w(\lambda_0)
	\end{equation}  for all $\lambda>\lambda_0$ provided that $\lambda_0>1$ is sufficiently large. On the other hand, since $u\in C^{2,\alpha}_\tau (\mathbb{R}^n_+)$, we have
	$$
	w(\lambda)\geq -O(\lambda^{(n-2)-\tau}).
	$$
Note that this is not compatible with \eqref{w upper estimate}. This completes the proof in the case where $\Delta_g u=0$ in $\mathbb{R}^n_+\setminus B^n_{\lambda_0}(0)$. \\ \indent
Now, suppose that $\Delta_g u\geq 0$ in $\mathbb{R}^n_+\setminus B^n_{\lambda_0}(0)$ and that $\Delta_gu$ is integrable. By Proposition \ref{PDE prop 2}, there is $v\in C^{2,\alpha}_{\tau}(\mathbb{R}^n_+)$ with
$$
\begin{dcases}\qquad \Delta_gv=0\qquad\qquad\qquad  &\text{in } \mathbb{R}^n_+\setminus B^n_{\lambda_0}(0),\qquad \qquad \qquad \qquad \qquad \\
\qquad D(g)_{\nu(\partial M,g)}v=0&\text{on } \partial \mathbb{R}^n_+\setminus B^n_{\lambda_0}(0),\text{ and} \\
\qquad v=u&\text{on }\mathbb{R}^n_+\cap\partial  B^n_{\lambda_0}(0).
\end{dcases}
$$
By the maximum principle, $u\leq v$ in $\mathbb{R}^n_+\setminus B^n_{\lambda_0}(0)$. In particular,
$$
\lambda^{1-n}\,\int_{\mathbb{R}^n_+\cap S^{n-1}_\lambda(0)}(u-v)\,\text{d}\bar\mu(g)\leq 0
$$
for every $\lambda\geq\lambda_0$. Moreover, using that $u,v\in C^{2,\alpha}_{\tau}(\mathbb{R}^n_+),$ we have
$$
\lim_{\lambda\to\infty}\lambda^{1-n}\,\int_{\mathbb{R}^n_+\cap S^{n-1}_\lambda(0)}(u-v)\,\text{d}\bar\mu(g)= 0.
$$
Consequently,
$$
\limsup_{\lambda\to\infty} \lambda^{-1}\,\sum_{i=1}^n\int_{\mathbb{R}^n_+\cap S^{n-1}_\lambda(0)}x^i\,\partial_i(u-v)\,\mathrm{d}\mu(\bar g)\geq 0.
$$
We have already showed that
$$
\lim_{\lambda\to\infty} \lambda^{-1}\,\sum_{i=1}^n\int_{\mathbb{R}^n_+\cap S^{n-1}_\lambda(0)}x^i\,\partial_iv\,\mathrm{d}\mu(\bar g)>0.
$$ 
Moreover, since $\Delta_g u$ is integrable, the argument that led to \eqref{integral exists} shows that
$$
\lim_{\lambda\to\infty} \lambda^{-1}\,\sum_{i=1}^n\int_{\mathbb{R}^n_+\cap S^{n-1}_\lambda(0)}x^i\,\partial_iu\,\mathrm{d}\mu(\bar g)
$$
exists.
\\ \indent The assertion follows. 
\end{proof} 

\section{Scalar curvature rigidity results} \label{scalar appendix}
In this section, we give an overview of several techniques that have been used to derive scalar curvature rigidity results in mathematical relativity. \\ \indent 
The  proofs given by R.~Schoen and S.-T.~Yau in \cite[Theorem 1]{schoenyautorus} and, independently, by M.~Gromov and H.~Lawson in \cite[Corollary A]{gromovlawson} of the following result are in some sense a precursor to the positive mass theorem, stated here as Theorem \ref{pmt no boundary}.
\begin{thm}[\cite{schoenyautorus,gromovlawson}] Let $n\geq 3$ be an integer and $T^n=S^1\times\cdots\times S^1$ be the torus of dimension $n$. Let $g$ be a Riemannian metric on $T^n$ with $R(g)\geq0$. Then $g$ is flat.  \label{torus rigidity} 
\end{thm}
The proofs in \cite{schoenyautorus,gromovlawson} show that any Riemannian metric $g$ on $T^n$ with $R(g)\geq 0$ must actually satisfy $R(g)=0$. Studying the variation of scalar curvature \eqref{kazdanwarner}, J.~P.~Bourguignon had previously observed that, unless $\operatorname{Ric}(g)=0$, such a metric can be perturbed to a metric of positive scalar curvature; see \cite[Lemma 5.2]{kazdanwarner}.  \\ \indent 
Let $(\tilde M,\tilde g)$ be an asymptotically flat manifold with $R(\tilde g)\geq 0$ and $m(\tilde g)=0$; see Appendix \ref{af appendix}.
The rigidity statement in the positive mass theorem, stated here as Theorem \ref{pmt no boundary}, can be viewed as a generalization of Theorem \ref{torus rigidity} to non-compact spaces.  By constructing a global variation of $\tilde g$, R.~Schoen and S.-T.~Yau have showed that, unless $\operatorname{Ric}(\tilde g)=0$, $\tilde g$ can be perturbed to a metric of non-negative scalar curvature and negative mass; see \cite[\S3]{SchoenYau}.  If $(\tilde M,\tilde g)$ is spin, the alternative argument of E.~Witten implies the existence of certain parallel spinors. The existence of these spinors implies that $(\tilde M,\tilde g)$ is flat; see \cite[\S3]{Witten}. We note that Y.~Shi and L.-F.~Tam have adapted this argument to settings with lower regularity; see \cite[\S3]{ShiTam}. Recently, S.~Lu and P.~Miao \cite[Proposition 2.1]{LuMiao} have extended the rigidity statement in Theorem \ref{pmt no boundary} to metrics with a corner. Their proof uses an argument of D.~McFeron and G.~Sz\'{e}kelyhidi \cite{McFeron} based on the observation that the mass is constant along Ricci flow. \\ \indent 
If $n=3$, the proof of the positive mass theorem in \cite{SchoenYau} suggests that there are no non-compact properly embedded   area-minimizing surfaces in $\tilde M$ unless $(\tilde M,\tilde g)$ is flat $\mathbb{R}^3$. This conjecture of R.~Schoen has been confirmed by O.~Chodosh and the first-named author in \cite[Theorem 1.6]{CCE}. In the proof, they use local perturbations of $\tilde g$ to construct a local foliation of a neighborhood of such a minimal surface by non-compact properly embedded  area-minimizing  surfaces obtained as limits of solutions of the Plateau problem. We remark that related rigidity results that restrict the topology of a horizon boundary are known; see \cite[Corollary 1.4]{IDRR} and the references therein.   \\ \indent 
Let $(\tilde M,\tilde g)$ be asymptotically flat of dimension $3\leq n\leq7$ with horizon boundary $\tilde \Sigma\subset \tilde M$. If $n=3$ and $\tilde \Sigma$ is connected, G.~Huisken and T.~Ilmanen have proved the Riemannian Penrose inequality, stated here as Theorem \ref{RPI no boundary}, by evolving $\tilde \Sigma$ by inverse mean curvature flow to a large coordinate sphere in the asymptotically flat chart. By an explicit calculation, they have showed that the Hawking mass of the evolving horizon is non-decreasing and, in fact, constant if $(\tilde M(\tilde \Sigma),\tilde g)$ is scalar flat and foliated by totally umbilic constant mean curvature spheres. To prove Theorem \ref{RPI no boundary} in the general case where $3\leq n\leq7$ and $\tilde \Sigma$ is possibly disconnected, H.~L.~Bray and D.~A.~Lee have used a conformal flow of the metric $\tilde g$ along which the mass is non-increasing while the area of the horizon boundary remains constant. Their proof shows that the mass is, in fact, constant along this flow if and only if a suitable conformal transformation of the double of $\tilde M(\tilde \Sigma)$ obtained by reflection across $\tilde \Sigma$ has zero mass. If $(\tilde M,\tilde g)$ is spin, the rigidity results in \cite[\S3]{ShiTam} apply to the possibly non-smooth double. As a consequence, $(\tilde M(\tilde \Sigma),\tilde g)$ is isometric to the exterior region of a Schwarzschild space \eqref{Schwarzschild space}. The work of S.~Lu and P.~Miao \cite{LuMiao} shows that the assumption that $(\tilde M,\tilde g)$ be spin is not necessary.\\ \indent 
Finally, in Theorem \ref{RPI rigidity}, we give a short variational proof of rigidity in the Riemannian Penrose inequality that does not require the spin assumption. To this end, we show that if equality holds in \eqref{RPI}, then the double of $(\tilde M(\tilde \Sigma),\tilde g)$ obtained by reflection across $\tilde \Sigma$ is smooth. In fact, we observe that if $\tilde \Sigma$ has non-vanishing second fundamental form, we can locally perturb the metric $\tilde g$ to increase $R(\tilde g)$ without decreasing the area of $\tilde \Sigma$. By a global conformal transformation to zero scalar curvature, we may then decrease $\tilde m(\tilde g)$ without  decreasing the area of $\tilde \Sigma$ by much. 
\end{appendices}


\begin{bibdiv}
	\begin{biblist}
		
		\bib{Almaraz}{article}{
			author={Almaraz, S\'{e}rgio},
			title={Convergence of scalar-flat metrics on manifolds with boundary
				under a {Y}amabe-type flow},
			date={2015},
			ISSN={0022-0396},
			journal={J. Differential Equations},
			volume={259},
			number={7},
			pages={2626\ndash 2694},
			url={https://doi.org/10.1016/j.jde.2015.04.011},
			review={\MR{3360653}},
		}
		
		\bib{ABDL}{article}{
			author={Almaraz, S\'{e}rgio},
			author={Barbosa, Ezequiel},
			author={de~Lima, Levi~Lopes},
			title={A positive mass theorem for asymptotically flat manifolds with a
				non-compact boundary},
			date={2016},
			ISSN={1019-8385},
			journal={Comm. Anal. Geom.},
			volume={24},
			number={4},
			pages={673\ndash 715},
			url={https://doi.org/10.4310/CAG.2016.v24.n4.a1},
			review={\MR{3570413}},
		}
		
		\bib{Luciano}{article}{
			author={Almaraz, S\'{e}rgio},
			author={de~Lima, Levi~Lopes},
			author={Mari, Luciano},
			title={Spacetime positive mass theorems for initial data sets with
				non-compact boundary},
			date={2021},
			ISSN={1073-7928},
			journal={Int. Math. Res. Not. IMRN},
			number={4},
			pages={2783\ndash 2841},
			url={https://doi.org/10.1093/imrn/rnaa226},
			review={\MR{4218338}},
		}
		
		\bib{ADM}{article}{
			author={Arnowitt, Richard},
			author={Deser, Stanley},
			author={Misner, Charles},
			title={Coordinate invariance and energy expressions in general
				relativity},
			date={1961},
			ISSN={0031-899X},
			journal={Phys. Rev. (2)},
			volume={122},
			pages={997\ndash 1006},
			review={\MR{127946}},
		}
		
		\bib{BarbosaMeira}{article}{
			author={Barbosa, Ezequiel},
			author={Meira, Adson},
			title={A positive mass theorem and {P}enrose inequality for graphs with
				noncompact boundary},
			date={2018},
			ISSN={0030-8730},
			journal={Pacific J. Math.},
			volume={294},
			number={2},
			pages={257\ndash 273},
			url={https://doi.org/10.2140/pjm.2018.294.257},
			review={\MR{3770113}},
		}
		
		\bib{Bartnik}{article}{
			author={Bartnik, Robert},
			title={The mass of an asymptotically flat manifold},
			date={1986},
			ISSN={0010-3640},
			journal={Comm. Pure Appl. Math.},
			volume={39},
			number={5},
			pages={661\ndash 693},
			url={https://doi.org/10.1002/cpa.3160390505},
			review={\MR{849427}},
		}
		
		\bib{Bray}{article}{
			author={Bray, Hubert~L.},
			title={Proof of the {R}iemannian {P}enrose inequality using the positive
				mass theorem},
			date={2001},
			ISSN={0022-040X},
			journal={J. Differential Geom.},
			volume={59},
			number={2},
			pages={177\ndash 267},
			url={http://projecteuclid.org/euclid.jdg/1090349428},
			review={\MR{1908823}},
		}
		
		\bib{BrayLee}{article}{
			author={Bray, Hubert~L.},
			author={Lee, Dan~A.},
			title={On the {R}iemannian {P}enrose inequality in dimensions less than
				eight},
			date={2009},
			ISSN={0012-7094},
			journal={Duke Math. J.},
			volume={148},
			number={1},
			pages={81\ndash 106},
			url={https://doi.org/10.1215/00127094-2009-020},
			review={\MR{2515101}},
		}
		
		\bib{BrendleAsian}{article}{
			author={Brendle, Simon},
			title={A generalization of the {Y}amabe flow for manifolds with
				boundary},
			date={2002},
			ISSN={1093-6106},
			journal={Asian J. Math.},
			volume={6},
			number={4},
			pages={625\ndash 644},
			url={https://doi.org/10.4310/AJM.2002.v6.n4.a2},
			review={\MR{1958085}},
		}
		
		\bib{BrendleChen}{article}{
			author={Brendle, Simon},
			author={Chen, Szu-Yu~Sophie},
			title={An existence theorem for the {Y}amabe problem on manifolds with
				boundary},
			date={2014},
			ISSN={1435-9855},
			journal={J. Eur. Math. Soc. (JEMS)},
			volume={16},
			number={5},
			pages={991\ndash 1016},
			url={https://doi.org/10.4171/JEMS/453},
			review={\MR{3210959}},
		}
		
		\bib{CCE}{article}{
			author={Carlotto, Alessandro},
			author={Chodosh, Otis},
			author={Eichmair, Michael},
			title={Effective versions of the positive mass theorem},
			date={2016},
			ISSN={0020-9910},
			journal={Invent. Math.},
			volume={206},
			number={3},
			pages={975\ndash 1016},
			url={https://doi.org/10.1007/s00222-016-0667-3},
			review={\MR{3573977}},
		}
		
		\bib{CarlottoSchoen}{article}{
			author={Carlotto, Alessandro},
			author={Schoen, Richard},
			title={Localizing solutions of the {E}instein constraint equations},
			date={2016},
			ISSN={0020-9910},
			journal={Invent. Math.},
			volume={205},
			number={3},
			pages={559\ndash 615},
			url={https://doi.org/10.1007/s00222-015-0642-4},
			review={\MR{3539922}},
		}
		
		\bib{Corvino}{article}{
			author={Corvino, Justin},
			title={Scalar curvature deformation and a gluing construction for the
				{E}instein constraint equations},
			date={2000},
			ISSN={0010-3616},
			journal={Comm. Math. Phys.},
			volume={214},
			number={1},
			pages={137\ndash 189},
			url={https://doi.org/10.1007/PL00005533},
			review={\MR{1794269}},
		}
		
		\bib{czimek2014static}{thesis}{
			author={Czimek, Stefan},
			title={On the static metric extension problem},
			type={Master's thesis},
			date={2014},
			note={\url{https://www.math.uni-leipzig.de/~czimek/mthesis.pdf}},
		}
		
		\bib{deLima}{article}{
			author={de~Lima, Levi~Lopes},
			title={Conserved quantities in general relativity: the case of initial
				data sets with a noncompact boundary.},
			date={2021},
			journal={arXiv preprint arXiv:2103.06061},
			url={https://arxiv.org/abs/2103.06061},
			note={to appear in Perspectives in Scalar Curvature, edited by M.
				Gromov and H.B. Lawson, Jr. World Scientific, 2022.},
		}
		
		\bib{IDRR}{article}{
			author={Eichmair, Michael},
			author={Galloway, Gregory~J.},
			author={Mendes, Abra\~{a}o},
			title={Initial data rigidity results},
			date={2021},
			ISSN={0010-3616},
			journal={Comm. Math. Phys.},
			volume={386},
			number={1},
			pages={253\ndash 268},
			url={https://doi.org/10.1007/s00220-021-04033-x},
			review={\MR{4287186}},
		}
		
		\bib{Escobar}{article}{
			author={Escobar, Jos\'{e}~F.},
			title={The {Y}amabe problem on manifolds with boundary},
			date={1992},
			ISSN={0022-040X},
			journal={J. Differential Geom.},
			volume={35},
			number={1},
			pages={21\ndash 84},
			url={http://projecteuclid.org/euclid.jdg/1214447805},
			review={\MR{1152225}},
		}
		
		\bib{GT}{book}{
			author={Gilbarg, David},
			author={Trudinger, Neil~S.},
			title={Elliptic partial differential equations of second order},
			edition={Second},
			series={Grundlehren der mathematischen Wissenschaften [Fundamental
				Principles of Mathematical Sciences]},
			publisher={Springer-Verlag, Berlin},
			date={1983},
			volume={224},
			ISBN={3-540-13025-X},
			url={https://doi.org/10.1007/978-3-642-61798-0},
			review={\MR{737190}},
		}
		
		\bib{gromovlawson}{article}{
			author={Gromov, Mikhael},
			author={Lawson, Herbert~Blaine},
			title={Spin and scalar curvature in the presence of a fundamental group.
				{I}},
			date={1980},
			ISSN={0003-486X},
			journal={Ann. of Math. (2)},
			volume={111},
			number={2},
			pages={209\ndash 230},
			url={https://doi.org/10.2307/1971198},
			review={\MR{569070}},
		}
		
		\bib{HuangWu}{article}{
			author={Huang, Lan-Hsuan},
			author={Wu, Damin},
			title={The equality case of the {P}enrose inequality for asymptotically
				flat graphs},
			date={2015},
			ISSN={0002-9947},
			journal={Trans. Amer. Math. Soc.},
			volume={367},
			number={1},
			pages={31\ndash 47},
			url={https://doi.org/10.1090/S0002-9947-2014-06090-X},
			review={\MR{3271252}},
		}
		
		\bib{HI}{article}{
			author={Huisken, Gerhard},
			author={Ilmanen, Tom},
			title={The inverse mean curvature flow and the {R}iemannian {P}enrose
				inequality},
			date={2001},
			ISSN={0022-040X},
			journal={J. Differential Geom.},
			volume={59},
			number={3},
			pages={353\ndash 437},
			url={http://projecteuclid.org/euclid.jdg/1090349447},
			review={\MR{1916951}},
		}
		
		\bib{kazdanwarner}{inproceedings}{
			author={Kazdan, Jerry~L.},
			author={Warner, Frank~W.},
			title={Prescribing curvatures},
			date={1975},
			booktitle={Differential geometry ({P}roc. {S}ympos. {P}ure {M}ath., {V}ol.
				{XXVII}, {S}tanford {U}niv., {S}tanford, {C}alif., 1973), {P}art 2},
			pages={309\ndash 319},
			review={\MR{0394505}},
		}
		
		\bib{thomas_penrose}{article}{
			author={Koerber, Thomas},
			title={The {R}iemannian {P}enrose inequality for asymptotically flat
				manifolds with non-compact boundary},
			date={2019},
			journal={arXiv preprint arXiv:1909.13283},
			url={https://arxiv.org/abs/1909.13283},
			note={to appear in J. Differential Geom.},
		}
		
		\bib{Lam}{book}{
			author={Lam, Mau-Kwong~George},
			title={The {G}raph {C}ases of the {R}iemannian {P}ositive {M}ass and
				{P}enrose {I}nequalities in {A}ll {D}imensions},
			publisher={ProQuest LLC, Ann Arbor, MI},
			date={2011},
			ISBN={978-1124-63054-0},
			url={http://gateway.proquest.com/openurl?url_ver=Z39.88-2004&rft_val_fmt=info:ofi/fmt:kev:mtx:dissertation&res_dat=xri:pqdiss&rft_dat=xri:pqdiss:3454195},
			note={Thesis (Ph.D.)--Duke University},
			review={\MR{2873434}},
		}
		
		\bib{LeeParker}{article}{
			author={Lee, John~M.},
			author={Parker, Thomas~H.},
			title={The {Y}amabe problem},
			date={1987},
			ISSN={0273-0979},
			journal={Bull. Amer. Math. Soc. (N.S.)},
			volume={17},
			number={1},
			pages={37\ndash 91},
			url={https://doi.org/10.1090/S0273-0979-1987-15514-5},
			review={\MR{888880}},
		}
		
		\bib{LuMiao}{article}{
			author={Lu, Siyuan},
			author={Miao, Pengzi},
			title={Rigidity of {R}iemannian {P}enrose inequality with corners and
				its implications},
			date={2021},
			ISSN={0022-1236},
			journal={J. Funct. Anal.},
			volume={281},
			number={10},
			pages={Paper No. 109231, 11},
			url={https://doi.org/10.1016/j.jfa.2021.109231},
			review={\MR{4309448}},
		}
		
		\bib{Marquardt}{article}{
			author={Marquardt, Thomas},
			title={Weak solutions of inverse mean curvature flow for hypersurfaces
				with boundary},
			date={2017},
			ISSN={0075-4102},
			journal={J. Reine Angew. Math.},
			volume={728},
			pages={237\ndash 261},
			url={https://doi.org/10.1515/crelle-2014-0116},
			review={\MR{3668996}},
		}
		
		\bib{MiaoCormick}{article}{
			author={McCormick, Stephen},
			author={Miao, Pengzi},
			title={On a {P}enrose-like inequality in dimensions less than eight},
			date={2019},
			ISSN={1073-7928},
			journal={Int. Math. Res. Not. IMRN},
			number={7},
			pages={2069\ndash 2084},
			url={https://doi.org/10.1093/imrn/rnx181},
			review={\MR{3938317}},
		}
		
		\bib{McFeron}{article}{
			author={McFeron, Donovan},
			author={Sz\'{e}kelyhidi, G\'{a}bor},
			title={On the positive mass theorem for manifolds with corners},
			date={2012},
			ISSN={0010-3616},
			journal={Comm. Math. Phys.},
			volume={313},
			number={2},
			pages={425\ndash 443},
			url={https://doi.org/10.1007/s00220-012-1498-8},
			review={\MR{2942956}},
		}
		
		\bib{meeks1982existence}{article}{
			author={Meeks, William~W., III},
			author={Yau, Shing-Tung},
			title={The existence of embedded minimal surfaces and the problem of
				uniqueness},
			date={1982},
			ISSN={0025-5874},
			journal={Math. Z.},
			volume={179},
			number={2},
			pages={151\ndash 168},
			url={https://doi.org/10.1007/BF01214308},
			review={\MR{645492}},
		}
		
		\bib{Miao}{article}{
			author={Miao, Pengzi},
			title={Positive mass theorem on manifolds admitting corners along a
				hypersurface},
			date={2002},
			ISSN={1095-0761},
			journal={Adv. Theor. Math. Phys.},
			volume={6},
			number={6},
			pages={1163\ndash 1182 (2003)},
			url={https://doi.org/10.4310/ATMP.2002.v6.n6.a4},
			review={\MR{1982695}},
		}
		
		\bib{Penrosenaked}{article}{
			author={Penrose, Roger},
			title={Naked singularities},
			date={1973},
			journal={Annals of the New York Academy of Sciences},
			volume={224},
			number={1},
			pages={125\ndash 134},
			eprint={https://nyaspubs.onlinelibrary.wiley.com/doi/pdf/10.1111/j.1749-6632.1973.tb41447.x},
			url={https://nyaspubs.onlinelibrary.wiley.com/doi/abs/10.1111/j.1749-6632.1973.tb41447.x},
		}
		
		\bib{Raulot}{article}{
			author={Raulot, Simon},
			title={Green functions for the {D}irac operator under local boundary
				conditions and applications},
			date={2011},
			ISSN={0232-704X},
			journal={Ann. Global Anal. Geom.},
			volume={39},
			number={4},
			pages={337\ndash 359},
			url={https://doi.org/10.1007/s10455-010-9236-y},
			review={\MR{2776767}},
		}
		
		\bib{SchoenYamabe}{article}{
			author={Schoen, Richard},
			title={Conformal deformation of a {R}iemannian metric to constant scalar
				curvature},
			date={1984},
			ISSN={0022-040X},
			journal={J. Differential Geom.},
			volume={20},
			number={2},
			pages={479\ndash 495},
			url={http://projecteuclid.org/euclid.jdg/1214439291},
			review={\MR{788292}},
		}
		
		\bib{schoenvariational}{incollection}{
			author={Schoen, Richard},
			title={Variational theory for the total scalar curvature functional for
				{R}iemannian metrics and related topics},
			date={1989},
			booktitle={Topics in calculus of variations ({M}ontecatini {T}erme, 1987)},
			series={Lecture Notes in Math.},
			volume={1365},
			publisher={Springer, Berlin},
			pages={120\ndash 154},
			url={https://doi.org/10.1007/BFb0089180},
			review={\MR{994021}},
		}
		
		\bib{SchoenYau}{article}{
			author={Schoen, Richard},
			author={Yau, Shing-Tung},
			title={On the proof of the positive mass conjecture in general
				relativity},
			date={1979},
			ISSN={0010-3616},
			journal={Comm. Math. Phys.},
			volume={65},
			number={1},
			pages={45\ndash 76},
			url={http://projecteuclid.org/euclid.cmp/1103904790},
			review={\MR{526976}},
		}
		
		\bib{schoenyautorus}{article}{
			author={Schoen, Richard},
			author={Yau, Shing-Tung},
			title={On the structure of manifolds with positive scalar curvature},
			date={1979},
			ISSN={0025-2611},
			journal={Manuscripta Math.},
			volume={28},
			number={1-3},
			pages={159\ndash 183},
			url={https://doi.org/10.1007/BF01647970},
			review={\MR{535700}},
		}
		
		\bib{ShiTam}{article}{
			author={Shi, Yuguang},
			author={Tam, Luen-Fai},
			title={Positive mass theorem and the boundary behaviors of compact
				manifolds with nonnegative scalar curvature},
			date={2002},
			ISSN={0022-040X},
			journal={J. Differential Geom.},
			volume={62},
			number={1},
			pages={79\ndash 125},
			url={http://projecteuclid.org/euclid.jdg/1090425530},
			review={\MR{1987378}},
		}
		
		\bib{volkmann2015free}{thesis}{
			author={Volkmann, Alexander},
			title={Free boundary problems governed by mean curvature},
			type={{P}h{D} thesis},
			date={2015},
			note={\url{https://d-nb.info/1067442340/34}},
		}
		
		\bib{Witten}{article}{
			author={Witten, Edward},
			title={A new proof of the positive energy theorem},
			date={1981},
			ISSN={0010-3616},
			journal={Comm. Math. Phys.},
			volume={80},
			number={3},
			pages={381\ndash 402},
			url={http://projecteuclid.org/euclid.cmp/1103919981},
			review={\MR{626707}},
		}
		
	\end{biblist}
\end{bibdiv}
\end{document}